\documentclass{article}
\usepackage[left=20mm,right=20mm,top=35mm]{geometry} 
\usepackage[utf8]{inputenc}
\usepackage{physics}
\usepackage[nottoc]{tocbibind}
\usepackage{amsthm}
\usepackage{tikz}
\usepackage[colorlinks=true]{hyperref}
\usepackage{authblk}
\usepackage{amsfonts}
\usepackage{caption}
\usepackage{subcaption}
\usepackage{wrapfig}
\usepackage{longtable}
\usepackage{xcolor,comment}
\usepackage{amsmath,mathtools}
\usepackage{graphicx}

\newtheorem{theorem}{Theorem}

\newtheorem{lemma}[theorem]{Lemma}
\newtheorem{definition}{Definition}
\newtheorem{example}{Example}
\newtheorem{proposition}{Proposition}

\newcommand\sbullet[1][.5]{\mathbin{\vcenter{\hbox{\scalebox{#1}{$\bullet$}}}}}
\newcommand{\singular}{^{\sbullet[0.8]}}

\begin{document}

\title{A perturbed Alexander polynomial for singular knots}
\author[1, 2]{\normalsize{Martine Schut} \thanks{Correspondence email: \href{mailto:martine.schut@rug.nl}{martine.schut@rug.nl}}}
\author[1]{\normalsize{Roland van der Veen}}
\affil[1]{\small{Bernoulli Institute for Mathematics, Computing Science and Artificial Intelligence, University of Groningen, 9747AG Groningen, the Netherlands}}
\affil[2]{\small{Van Swinderen Institute for Particle Physics and Gravity, University of Groningen, 9747AG Groningen, the Netherlands}}
\date{\normalsize{\today}}

\maketitle

\begin{abstract}
We introduce a version of the Alexander polynomial for singular knots and tangles and show how it can be strengthened considerably by introducing a perturbation. For singular long knots, we also prove that our Alexander polynomial agrees with previous definitions.
\end{abstract}

\tableofcontents

\newpage
\section{Introduction}
We introduce a strong singular knot invariant that is easy to compute.
The invariant is closely related to the singular version of the Alexander polynomial.
Our strategy is based on the $\Gamma$-calculus developed by Bar-Natan~\cite{bar2015balloons,vo2018alexander} and Selmani~\cite{bar2013meta}; also see van der Veen and Mashaghi~\cite{mashaghi2021polynomial}.
The $\Gamma$-calculus can define the tangle diagrams in a recursive way such that the Alexander polynomial can easily be computed. 
Analogously, we define the singular $\Gamma$-calculus, denoted $\Gamma_\textup{s}$.
Our invariant works for tangles, and we start by introducing singular tangle diagrams as we use a slightly unconventional notation for tangle diagrams.
As the $\Gamma$-calculus produces the Alexander polynomial, the $\Gamma_\textup{s}$-calculus is expected to produce the singular version of the Alexander polynomial.
Furthermore, we propose an extension of the $\Gamma_\textup{s}$-calculus to construct a stronger invariant, which we refer to as the perturbed singular Alexander polynomial and denote $\rho_1^\textup{s}$. 
This invariant is obtained in a way analogously to the perturbed Alexander invariant $\rho_1$ introduced in ref.~\cite{bar2021perturbed,bar2022perturbed}. 
The invariant $\rho_1^\textup{s}$ is easy to compute and stronger than the singular Alexander polynomial. 
Appendix~\ref{app:knot_table} shows a singular knot table with knots up to five crossings with their singular Alexander and perturbed singular Alexander polynomial.
In studying proteins (which are folded chains of amino acids) or molecular folding, knot theory can be used; see e.g.~\cite{mashaghi2014circuit}.
A singular crossing could, for example, represent a rigid H-bond in these studies.

\section{The singular Gamma-calculus}\label{sec:sing_gamma}
We introduce the singular tangle diagrams along with their Reidemeister moves\index{tangle diagram} and connection to singular braids\index{Reidemeister moves!singular tangles}. 
Then we introduce the $\Gamma$-calculus for singular tangles.

\subsection{Singular tangle diagrams}

\begin{definition}
A singular tangle diagram $D$ is a directed graph with vertices carrying a cyclic ordering of the edges meeting there.
We allow two types of vertices: endpoints and crossings. At an endpoint, only one edge is present. At a crossing, precisely
two adjacent edges enter and two exit. The crossings are labelled by an element of the set $\{0,-1,1\}$ and are called singular, negative and positive, respectively.
The set of edges is assumed to be a disjoint union of directed paths, each starting and ending at distinct endpoints, going straight at each crossing.
These paths are known as the strands of the diagram. Each of the strands carries a distinct label, and the set of strand labels is denoted by  
$\mathcal{L}(D)$.
\end{definition}

A singular tangle diagram with a single strand is called a (singular) long knot diagram. 
The most basic examples of singular tangle diagrams are the three crossings (denoted $X^+_{ij},X^{-}_{ij}$ and $X{\singular}_{ij}$.) and the crossingless strand $1_i$ shown in fig.~\ref{fig:crossings}. 
Here, $i$ is the label of the overpassing strand, and $j$ is the label of the underpass.
In the singular case, we take $i$ as the label of the strand belonging to the left-most incoming edge.
%
\begin{figure}[htp!]
\begin{center}
\includegraphics[width=0.6\textwidth]{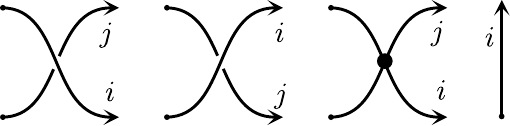}
\caption{From left to right: The positive crossing $X^{+}_{ij}$, the negative crossing $X^{-}_{ij}$, the singular crossing $X\singular_{ij}$ and the crossingless strand $1_i$.}
\label{fig:crossings}
\end{center}
\end{figure}
Any other singular tangle diagram can be obtained from the minimal ones from fig.~\ref{fig:crossings} using two operations: disjoint union and merging. 

\begin{definition}[Disjoint union]\label{def:disjoint-union}
    By disjoint union, we literally mean the disjoint union of the underlying graphs, assuming the strand label sets are distinct (or possibly renaming the strands to make sure the labels of the strands remain distinct).
\end{definition}
To save space, we will use juxtaposition to denote the disjoint union of diagrams.
For example, $X_{ij}^+X_{kl}\singular$ is the disjoint union of a positive crossing with understrand $i$ and a singular crossing with left-most strand $k$.

\begin{definition}[Merging]\label{def:merging}
    Given a singular tangle diagram $D$ and $i\neq j\in \mathcal{L}(D)$ and $k\notin \mathcal{L}(D)$, define a new diagram $m^{ij}_k(D)$ as follows:
    Suppose that in $D$ the final edge of strand $i$ 
    with endpoints $(v,w)$ and suppose that the first edge of strand $j$ is $(x,y)$. 
    The diagram $m^{ij}_k(D)$ is obtained from $D$ by deleting the edges $(v,w)$ and $(x,y)$ from $D$ and adding a new edge $(v,y)$. 
    The resulting strand is renamed $k$ so that
    $\mathcal{L}(m^{ij}_k(D)) = \mathcal{L}(D)\setminus\{i,j\}\cup \{k\}$.    
\end{definition}

Sometimes it is useful to merge many strands at the same time. 
For this we introduce the notation $m^{i,j}_k$ where $i,j,k$ are vectors of the same length $n$ and where the sets $\{i_1,\dots i_n\}$ and $\{j_1,\dots j_n\}$ are disjoint.
The notation $m^{i,j}_k$ then means $m^{i,j}_k = m^{i_1,j_1}_{k_1} \circ m^{i_2,j_2}_{k_2} \circ \dots m^{i_n,j_n}_{k_n}$, i.e. the merging of ordered sets.
Also, $m^{i}_k$ is assumed to mean $m^{1,k}_{k}\circ m^{2,k}_{k}\circ \dots \circ m^{i_{n-2},k}_{k}\circ m^{i_{n-1},i_n}_k$, i.e. the subsequent merging to the same strand.

Since they really are graphs with some additional structure, our tangle diagrams can easily be assembled from the basic diagrams shown above using these two operations.

\subsubsection*{Reidemeister moves}
Since we want to investigate tangle projections (resulting in 2-dimensional diagrams) of tangles (i.e. 3-dimensional objects), we present the Reidemeister moves on the tangle diagrams.
From Reidemeister's theorem~\cite{reidemeister2013knotentheorie}, if two tangle diagrams are equivalent, meaning that they can be related with a sequence of Reidemeister moves (and possibly renaming the labels), then they correspond to the same tangle.
Therefore, we use the Reidemeister moves presented below to describe the equivalence of tangle diagrams:
\begin{align}
m^{1,2}_1 X_{1,2}^\pm &= 1_1\ , \qquad m^{2,1}_1 X_{1,2}^\pm = 1_1 \label{eq:RM1} \\ 
m^{(1,2),(3,4)}_{1,2}(X^{+}_{12}X^{-}_{3,4}) &= 1_11_2 \label{eq:RM2} \\
   m^{(1,2,3),(4,5,6)}_{1,2,3}(X^{+}_{1,2} X^+_{4,3} X^{+}_{5,6}) &= m^{(1,2,3),(4,5,6)}_{1,2,3} (X^{+}_{1,6} X^+_{2,3} X^{+}_{4,5}) \label{eq:RM3} \\
   m^{(1,2),(4,3)}_{1,2}(X\singular_{1,2} X_{3,4}^+)  &= m^{(1,2),(4,3)}_{1,2}(X^+_{1,2} X_{3,4}\singular) \label{eq:sing_RM_rel1} \\
    m^{(1,2,3),(4,5,6)}_{1,2,3}(X^{-}_{2,1} X\singular_{4,3} X^{+}_{5,6}) &= m^{(1,2,3),(4,5,6)}_{1,2,3} (X^{+}_{2,3} X\singular_{1,6} X^{-}_{5,4}) \label{eq:sing_RM_rel2}
\end{align}

\begin{figure}[thb]
     \centering
     \begin{subfigure}[h]{0.24\textwidth}
         \centering
         \includegraphics[width=\linewidth]{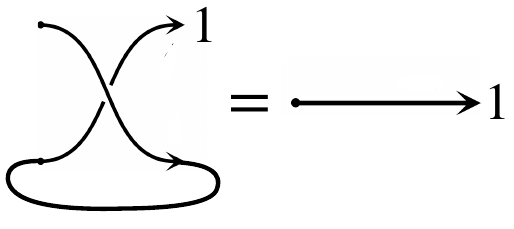}
         \caption{Reidemeister I, \newline $m^{1,2}_1 X_{1,2}^+ = 1_1$, see eq.~\eqref{eq:RM1}.}
         \label{fig-knots:RM1}
     \end{subfigure}
     \hfill
     \begin{subfigure}[h]{0.34\textwidth}
         \centering
         \includegraphics[width=\linewidth]{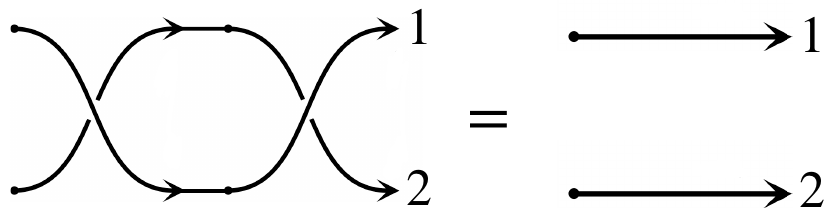}
         \caption{Reidemeister II in eq.~\eqref{eq:RM2}.}
         \label{fig-knots:RM2}
     \end{subfigure}
     \hfill
     \begin{subfigure}[h]{0.34\textwidth}
         \centering
         \includegraphics[width=\linewidth]{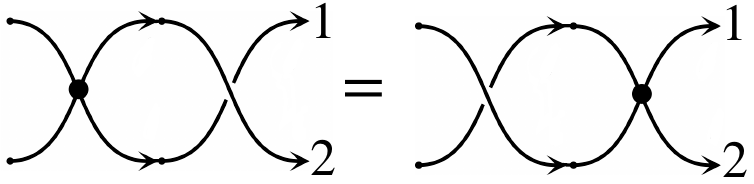}
         \caption{Singular Reidemeister II in eq.~\eqref{eq:sing_RM_rel1}.}
         \label{fig-knots:sRM2}
     \end{subfigure}
     \newline
     \begin{subfigure}[h]{0.45\textwidth}
         \centering
         \includegraphics[width=\linewidth]{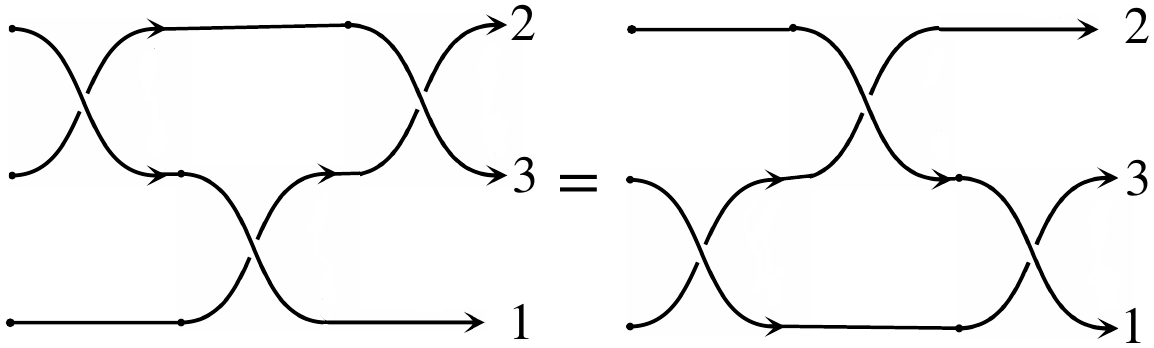}
         \caption{Reidemeister III in eq.~\eqref{eq:RM3}.}
         \label{fig-knots:RM3}
     \end{subfigure}
     \hfill
     \begin{subfigure}[h]{0.45\textwidth}
         \centering
         \includegraphics[width=\linewidth]{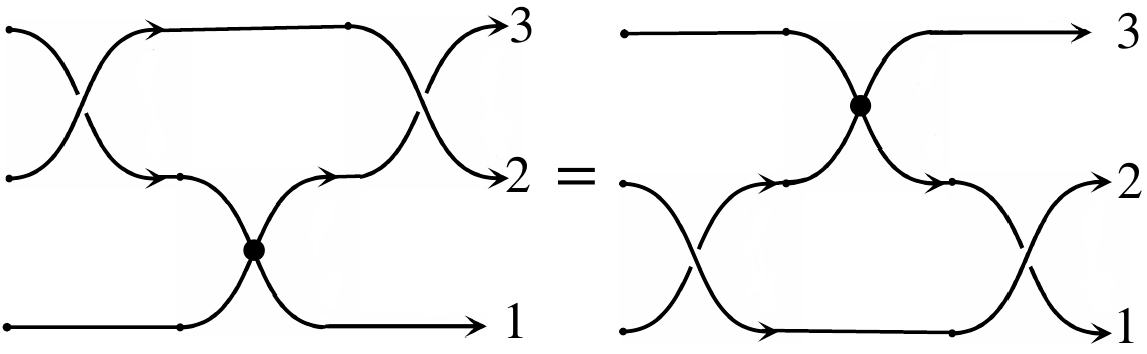}
         \caption{Singular Reidemeister III in eq.~\eqref{eq:sing_RM_rel2}.}
         \label{fig-knots:sRM3}
     \end{subfigure}
        \caption{}
        \label{fig-knots:RMmoves}
\end{figure}

\subsection{Singular gamma calculus}

We start by recalling the structure behind the $\Gamma$-calculus; see Selmani-Bar-Natan~\cite{bar2015balloons,vo2018alexander,bar2013meta} and ~\cite{mashaghi2021polynomial} for more background. 
Given a finite set $S$ (of strand labels), we consider quadratic polynomials in variables $r_i$, $c_i$ with coefficients in a field $F$, where $i$ runs through $S$.
Denote by $P_S$ the set of all pairs $(\omega,A)$ where $A = \sum_{i,j\in S} r_i A_{ij}c_j$ and $A_{ij},\omega \in F$. 

On the sets $P_S$, we construct two maps called disjoint union and merging as follows:
If $S$ and $S'$ are disjoint
we have a map $\sqcup:P_S\times P_{S'}\to P_{S\cup S'}$ defined by $(\omega,A)\sqcup (\omega',A') = (\omega \omega', A+A')$. 
To save space, we will use the shorthand notation $XY = X\sqcup Y$.
If $i,j\in S$ and $k\notin S\setminus\{i,j\}$, we also define $m^{ij}_k:P_S\to P_{S\setminus \{i,j\}\cup \{k\}}$ by
\begin{equation}
\label{eq.mijk}
m^{ij}_k(\omega,A) =\left(\omega \left( 1 - A_{ij}\right), A + \frac{A_{i,\bullet}A_{\bullet,j}}{1-A_{ij}}\right)\eval_{\substack{r_i,c_j \mapsto 0 \\ r_j, c_i \mapsto r_k, c_k}}    
\end{equation}
where we used the notation
$A_{i,\bullet} = \sum_{j}A_{ij}c_j$ and $A_{\bullet,j}=\sum_{i}r_iA_{ij}$.
For later use, we record a lemma for speeding up the computation of $\Gamma$ by merging several strands simultaneously.

 
\begin{lemma}\label{lemma:bulk_closure}
Suppose $g:T\to S$ is a bijection of finite sets. 
\[
m^{(g(s_1),g(s_2),\dots),(s_1,s_2,\dots)}_{s_1,s_2\dots}(\omega_1,\sum_{i,j\in S}A_{i,j}r_ic_j)(\omega_2,\sum_{m,n\in T}B_{m,n}r_mc_n) = 
(\omega_1\omega_2,\sum_{i,j,k\in T}B_{i,k}A_{g(k),g(j)}r_ic_{j} )
\]
Also,
\[
m^{(1,2,3,\dots n)}_1 (\omega,\sum_{ij} r_iA_{ij}c_j) = (\omega\det(I-A_{ij}|_{i<n, 1<j}),Zr_1c_1) 
\]
for some scalar $Z$.
\end{lemma}
\begin{proof}
Induction on the size of the index set $S$. See \cite{vo2018alexander} for details.
\end{proof}

\begin{definition}[$\Gamma_\textup{s}$-calculus]\label{def:pol_gamma}
Working over the field $F = \mathbb{Q}(s,t)$ of rational functions in variables $s,t$ define for any singular tangle diagram $D$ the Gamma invariant, $\Gamma(D)\in P_{\mathcal{D}}$ as follows:
\begin{enumerate}
\item $\Gamma_\textup{s}(1_i) = (1, r_i c_i)$
    \item The values of the crossings are \begin{align}\Gamma_\textup{s}(X_{ij}^{\pm}) &= (1, r_i c_i + (1-t^{\pm1}) r_i c_j + t^{\pm1} r_j c_j) \label{eq.regularX}\\
    \Gamma_\textup{s}(X\singular_{ij}) & = (1, s \, r_i c_i + (1-st) r_i c_j + s t \, r_j c_j + (1-s) r_j c_i)\label{eq.singularX}\end{align}
    \item $\Gamma_\textup{s}(m^{ij}_k D) = m^{ij}_k(D)$.
    \item $\Gamma_\textup{s}(DD') = \Gamma_\textup{s}(D)\Gamma_\textup{s}(D')$.
    \end{enumerate}
\end{definition}

\begin{figure}[tbh]
\begin{center}
\includegraphics[width=10cm]{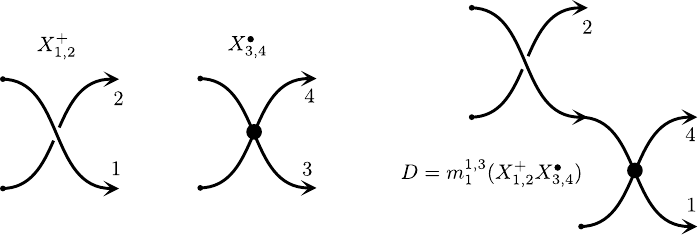}
\end{center}
\caption{}
\label{fig:example1}
\end{figure}

\begin{example}\label{ex:definition_gamma}
    For example, consider the stitching of two crossings $X^+_{1,2}$ and $X\singular_{3,4}$, such that the endpoint of strand $1$ is stitched to the starting point of strand $3$, to form the singular tangle diagram
    $D=m^{1,3}_1(X^+_{12}X\singular_{3,4})$
    shown in fig.~\ref{fig:example1}).
    The crossings are given by:
    \begin{align}
        \Gamma_\textup{s}(X_{1,2}^{+}) &= (1, r_1 c_1 + (1-t) r_1 c_2 + t \, r_2 c_2) \, , \\ 
        \Gamma_\textup{s}(X\singular_{3,4}) &= (1, s \, r_3 c_3 + (1-st) r_3 c_4 + s t \, r_4 c_4 + (1-s) r_4 c_3) \, .
    \end{align}
    The disjoint union of the two crossings is given by:
    \begin{align}
        \Gamma_\textup{s}(X_{1,2}^{+}) &\sqcup \Gamma_\textup{s}(X_{3,4}\singular) 
        = \Gamma_\textup{s}(X_{1,2}^{+} \sqcup X_{3,4}\singular) \nonumber \\
        &= \left( 1, r_1 c_1 + (1-t) r_1 c_2 + t \, r_2 c_2 + s \, r_3 c_3 + (1-st) r_3 c_4 + s t \, r_4 c_4 + (1-s) r_4 c_3 \right) \, .
    \end{align}
    The stitching $m^{1,3}_1$ gives:
    \begin{align}
        \Gamma_\textup{s}&(m^{1,3}_1 X_{1,2}^{+} X_{3,4}\singular)
        = m^{1,3}_1(X_{1,2}^{+} \sqcup X_{3,4}\singular) \nonumber \\
        &= \left( 1, s (1-t) r_1 c_2 + (1-s t) r_1 c_4 + s \, r_1 c_1 +  (s-1) (t-1) r_4 c_2 + s \, t\, r_4 c_4 + (1-s) r_4 c_1 + t\, r_2 c_2 \right)
    \end{align}
    which is the $\Gamma$-invariant of the diagram $D$ shown in fig.~\ref{fig:example1}.
\end{example}

In Definition~\ref{def:pol_gamma} the $r_i$ and $c_j$ can be seen as indicating the $i^\text{th}$ row and the $j^\text{th}$ column, thus one could present the $\Gamma_\textup{s}$-calculus also in explicit matrix representation. 
Except for changing matrices into quadratic expressions our definition of $\Gamma$ calculus is identical to those found in the literature~\cite{vo2018alexander}. 
To extend the calculus to include singular tangles, all one needs to do is provide an appropriate value for the singular crossing $X\singular$.
Our choice of $\Gamma(X\singular)$ shown in eq.~\eqref{eq.singularX} was obtained by starting with undetermined coefficients and solving them using the singular Reidemeister moves. 
We set 
\begin{equation}
\Gamma_\textup{s}(X\singular_{a,b}) = (1, s \, r_a c_a + f \, r_a c_b + g \, r_b c_a + h \, r_b c_b)\, ,
\end{equation}
and solve for the coefficients using the singular Reidemeister moves in eq.~\eqref{eq:sing_RM_rel1-ex} and~\eqref{eq:sing_RM_rel2-ex} (illustrated in figs.~ \ref{fig-knots:sRM2},~\ref{fig-knots:sRM3} respectively):
\begin{align}
    \Gamma_\textup{s}\left(m_{1,2}^{(1,2),(4,3)} X\singular_{1,2} X_{3,4}^+ \right) &= \Gamma_\textup{s}\left( m_{1,2}^{(1,2),(4,3)}X_{1,2}^+ X\singular_{3,4}\right)  \label{eq:sing_RM_rel1-ex} \\
    \Gamma_\textup{s}\left( m^{(1,2,3),(4,5,6)}_{1,2,3} X^{-}_{2,1} X\singular_{4,3} X^{+}_{5,6} \right) &= \Gamma_\textup{s}\left( m^{(1,2,3),(4,5,6)}_{1,2,3} X^{+}_{2,3} X\singular_{1,6} X^{-}_{5,4} \right) \label{eq:sing_RM_rel2-ex}
\end{align}
The two singular Reidemeister moves above provide, respectively, four and nine equations for the coefficients.
Solving them leaves one free variable:
$$ s = s \, , \qq{} f = 1 - s t \, , \qq{} g = 1 - s \,, \qq{} h = s t \, .$$
The coefficient $s$ can be chosen freely. 
We thus get the rule introduced for the singular crossing in Definition \ref{def:pol_gamma}.
The code used in this proof can be found on Github~\footnote{\href{https://github.com/MartineSchut/SingularKnots}{https://github.com/MartineSchut/SingularKnots}\label{footnote:github-knots}}.
We conclude that $\Gamma$ calculus yields an invariant of singular tangles. 
\begin{lemma}
$\Gamma(D)$ does not depend on the way $D$ was built from elementary tangles by merges and disjoint unions. Also, it is invariant under the singular Reidemeister moves
and invariant up to multiplication by $\pm t^{\pm}$ with respect to Reidemeister I.
\end{lemma}


\begin{lemma}\label{lemma:idk}
We say $(\omega,A)\in P_S$ is tangle-like if $A$ satisfies $\sum_{i\in S} A_{i,j} = 1$ for all $j\in S$.
In the matrix representation, this would mean that the sum of the row-coefficients equals unity.
If $(\omega,A)$ and  $(\omega',A')$ are tangle-like then so are $m^{ij}_{k}(\omega,A)$ and  $(\omega,A)(\omega',A')$. 
In other words, the property is conserved under the merging and disjoint union operations.
\end{lemma}
\begin{proof}
This follows directly from eq.~\eqref{eq.mijk}.
\end{proof}
\begin{lemma}
For any long singular knot $K$ with strand $i$ we have $\Gamma_\textup{s}(K) = (\omega, r_i c_i)$ for some $\omega = \omega(K)\in \mathbb{Q}(s,t)$.
\end{lemma}
\begin{proof}
This follows follows from Lemma~\ref{lemma:idk} because $\Gamma_\textup{s}(X^\pm_{ij})$ and $\Gamma_\textup{s}(X\singular_{ij})$ and $\Gamma_\textup{s}(1_i)$ are all tangle-like. 
\end{proof}
In fact, we will later show that $\omega(K)\in \mathbb{Z}[t,s,t^{-1}]$.
%
%
There is an ambiguity of $\pm t^{k}$, $k\in\mathbb{Z}$ in the scalar part. 
This is because gamma calculus is not strictly invariant under the first Reidemeister move. We have to multiply by either $1$ or $t^{\pm}$. 
The extended invariant in sec.~\ref{sec:perturbed} solves this ambiguity.

\begin{figure}[thb]
    \centering
    \includegraphics[width=0.7\linewidth]{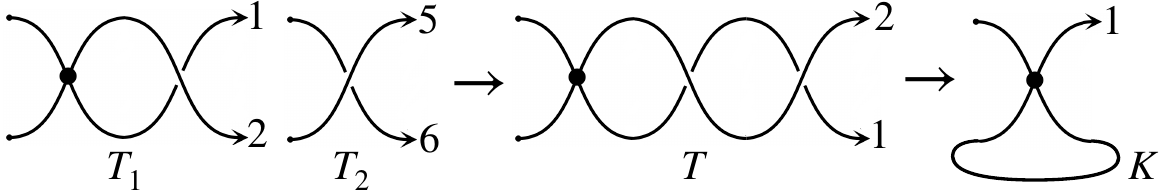}
    \caption{Diagrams corresponding to Example~\ref{ex:2}.}
    \label{fig-knots:ex2}
\end{figure}

\begin{example}\label{ex:2}
    As an example of Lemma~\ref{lemma:bulk_closure} we start with two tangle diagrams (depicted in fig.~\ref{fig-knots:ex2}): $T_1 = m^{(1,2),(4,3)}_{1,2} X\singular_{1,2} X_{3,4}^+$ and $T_2 = X_{5,6}^-$, and merge them together into $T = m^{(1,2),(6,5)}_{1,2}(T_1T_2)$. 
    Next, we find the $\Gamma_\textup{s}$ of the corresponding long knot $K=m^{1,2}_1(T)$.
    \begin{align}
        \Gamma_\textup{s}(T_1) &\equiv \left(1, \sum_{i,j\in S} A_{i,j} r_i c_j \right) = \left(1,r_2 \left(c_2 \left(s t^2-t+1\right)+c_1 (1-s t)\right)+r_1 \left(c_1 s t-c_2 t (s t-1)\right)\right) \, ,\\
        \Gamma_\textup{s}(T_2) &\equiv \left(1, \sum_{i,j\in T} B_{i,j} r_i c_j \right) = \left(1,r_5 c_6 \left(1-t^{-1}\right)+ t^{-1} r_6 c_6 +r_5 c_5\right)\, .
    \end{align}
    With the finite sets defined as $S=\{1,2\}$ and $T=\{5,6\}$.
    From Lemma~\ref{lemma:bulk_closure} we find that the composition of these two tangles. 
    Note that the function $g$ in Lemma~\ref{lemma:bulk_closure} is $(6,5)\mapsto(1,2)$.
    The composition of the two tangles via $m^{(1,2),(6,5)}_{1,2}$ is found:
    \begin{align}
        \Gamma_\textup{s}(T) 
        & = \Gamma_\textup{s}\left(m^{(1,2),(6,5)}_{1,2} m^{(1,2),(4,3)}_{1,2} X\singular_{1,2} X_{3,4}^+ X_{5,6}^- \right)
        = \left( 1\cdot 1, \sum_{i,j,k\in T} B_{i,k} A_{g(k),g(j)} r_i c_j \right) \\
        &= \left. \left( 1, \sum_{j\in(1,2)} B_{5,6} A_{1,j} r_2 c_j + B_{6,6} A_{1,j} r_1 c_j + B_{5,5} A_{2,j} r_2 c_j \right) \right\vert_{(6,5)\mapsto(1,2)} \\
        &= \left(1, (1 - s t) r_1  c_2 + s \, r_1 c_1 + s\,t\, r_2 c_2 + (1 - s) r_2 c_1 \right) = (1,\sum_{i,j} r_i M_{ij} c_j) \, .
    \end{align}
    Now closing the braid to form a long knot via the stitching $m^{1,2}_1$ using Lemma~\ref{lemma:bulk_closure} gives:
    \begin{equation}
        \Gamma_\textup{s}(K) = \left( 1 \cdot \det(1 - M_{ij}|_{i<2,j>1}), Z r_1 c_1 \right) = \left( s\, t, s\, r_1 c_1 \right) \, .
    \end{equation}
    In this simple case, $1-M_{12}= s\,t$ and the factor $Z=s$ can be found using the stitching in eq.~\eqref{eq.mijk}.
    The scalar part of $\Gamma_\textup{s}$ will later be shown to be the Alexander polynomial, $\omega(K) = \Delta_K$, up to a factor $\pm t^{k}$, $k\in\mathbb{Z}$. 
    Taking the limit $s=1$, which corresponds to replacing the singular crossing with a positive crossing, we see that $\Gamma_\textup{s} = (t,r_1 c_1)$.
    The diagram becomes that of the unknot; this example illustrates that the scalar part of the invariant contains a factor $t$ due to the first Reidemeister move only being invariant up to a factor $t^{k}$.
\end{example}

\section{Connection to previous work}
Previous works on singular knot invariants include a singular Burau representation~\cite{gemein2001representations} and singular Skein relations for the Kauffman state sum~\cite{fiedler2010jones,ozsvath2009floer}.
These works connect their singular knot invariants to the singular Alexander polynomial.
In this section, we show the relation between the aforementioned invariants and the invariant from the $\Gamma_s$-calculus, thus verifying that it produces the singular Alexander polynomial.

\subsection{Burau representation}\label{subsec:gemein}
\subsubsection{Singular braids}
To describe singular braid diagrams, we use a pair $(D,\pi)$ where $D$ is a singular tangle diagram with $\mathcal{L}(D) = \{1,2\dots n\}$ and $\pi\in S_n$ is a permutation. 
For example, $(D,\pi)$ on the top left of fig.~\ref{fig.SingBr} is a five-strand singular braid diagram.

\begin{figure}[tph!]
\begin{center}
\includegraphics[width=10cm]{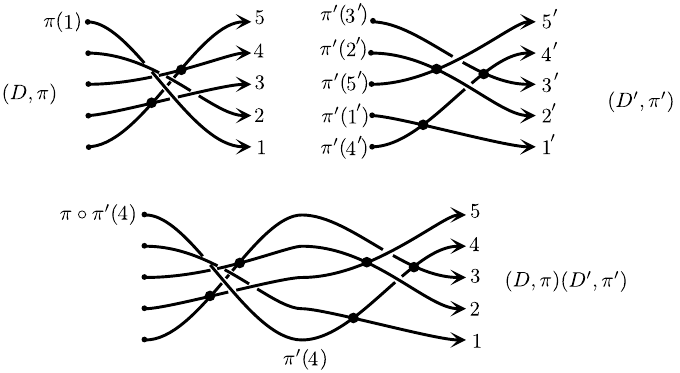}
\end{center}
\caption{Two 5-strand singular braid diagrams $(D,\pi)$ and $(D',\pi')$ and their composition $(D,\pi)(D',\pi')$.}
\label{fig.SingBr}
\end{figure}

Our convention is to number the heads of the strands of the braid diagram $1,2,3\dots, n$ starting at the bottom right. 
The permutation $\pi$ tells us the height of the tail of strand $i$; it is $\pi(i)$. 
In the example shown in fig.~\ref{fig.SingBr}, we showed that $\pi(1) = 5$.

Composing singular braid diagrams is done by concatenating them as shown at the bottom of fig.~\ref{fig.SingBr}. 
In terms of our singular tangle notation, we can describe this process as follows:
Whenever we have two pairs $(D,\pi)$ and $(D',\pi')$ where $\mathcal{L}(D)=\mathcal{L}(D') = \{1,\dots n\}$ and $\pi,\pi'\in S_n$ are permutations, define 
\begin{equation}\label{eq.braidcomp}
    (D,\pi)(D',\pi') = (m^{(\pi'(1),\pi'(2),\dots, \pi'(n)),(1',2',\dots, n')}_{(1\dots n)}(DD'),\pi\circ \pi') \, .
\end{equation} 
%
%
Here, we temporarily relabel the strands of $D'$ to be $1',2'\dots$ to avoid confusion.


\begin{theorem}[Alexander's theorem, singular version]\label{thm:alexander_thm}
For any (singular) knot $K$, there is a (singular) braid $b$ such that $K$ is isotopic to $K_b$.
\end{theorem}
\begin{proof}
    See~\cite{alexander1923lemma}. 
\end{proof}
It is known that two knots are equivalent if their corresponding long knots are equivalent~\cite{murasugi2012study}.
Therefore, we study braids and their closure to form a long knot. 

The Burau representation is a representation of the braid group $\mathcal{B}: B_n\to GL(n,Q(t))$, defined by 
\begin{equation}
    \mathcal{B}(\sigma_i) = \left(\begin{array}{cccc}I_{i-1} & 0 & 0 & 0\\
0& 1-t & t &0 \\ 0& 1 & 0 &0\\0&0&0&I_{n-i-1}\end{array}\right) \, , 
\end{equation}
where $I_k$ denotes the identity matrix of size $k$ and $\sigma_i$ the positive crossing generator on a strand labelled $i$.
The Burau representation is related to the Alexander polynomial of the braid closure by $\Delta(K_{b}) = \det_1(1-\mathcal{B}(b))$~\cite{murasugi2012study,BurdeZieschang}; the subscript $1$ denoting the removal of the first row and column.
In ref.~\cite{gemein2001representations} Gemein extended the braid group to a monoid of singular braids, $SB_n$, by introducing a new generator $\sigma\singular$.
They extended the Burau representation to a monoid homomorphism $\mathcal{B}:SB_n\to \mathrm{Mat}(n,\mathbb{Q}(t))$ by setting 
\begin{equation}
\mathcal{B}(\sigma_i\singular)=\left(\begin{array}{cccc}I_{i-1} & 0 & 0 & 0\\
0& 1-st & st & 0 \\ 0& s & 1-s& 0 \\0&0&0&I_{n-i-1}\end{array}\right) \, . 
\end{equation}
They showed that any singular knot can be presented as a closure $K_{b}$ of a singular braid $b$, and they defined the singular Alexander polynomial by
$G(K_{b}) = \det_1(1-\mathcal{B}(b))$.

\begin{theorem}
If $b = (D,\pi)$ is a singular braid and $\mathcal{B}_{ij}$ are the matrix entries of the Burau matrix of $b$, then 
\begin{equation}
    \Gamma_\textup{s}(D) = \left(1,\sum_{i,j}r_{i} \mathcal{B}_{j,i} c_{\pi^{-1}(j)} \right)
\end{equation}
\end{theorem}
\begin{proof}
For the crossings, this can be seen from the singular braid generators in eq.~\eqref{eq:braid-gen-pos}-\eqref{eq:braid-gen-sing}, showing that $\Gamma(X_{i,i+1}) = \left(1,\sum_{i,j} r_i \mathcal{B}_{i+1,i}(\sigma_i) c_{\pi^{-1}(i+1)} \right)$ and similarly for the singular crossing and \newline \noindent $\Gamma(X_{i+1,i}^-) = \left(1,\sum_{i,j} r_i \mathcal{B}_{i+1,i}(\sigma_i^{-1}) c_{\pi^{-1}(i+1)} \right)$
For two braid $b_1$, $b_2$ with corresponding diagram and permutation:
\begin{equation}
    \Gamma_\textup{s}(D_{1,2}) = \left(\omega_{1,2},\sum_{i,j}r_{i} \mathcal{B}_{j,i}(b_{1,2}) c_{\pi_{1,2}^{-1}(j)} \right) \text{with permutation } \pi_{1,2}\,.
\end{equation}
Concatenating the two braids - the endpoints of $D_1$, $\mathcal{L}(D_1)= \{1',\ldots,n'\}$ to the beginning points of $D_2$, $\mathcal{L}(D_2)=\{1,\ldots,n\}$ - via $m^{(\pi_2(1),\pi_2(2),\ldots,\pi_2(n)),(1,2,\ldots,n)}_{(1,\ldots,n)}$ gives the braid $b=(D,\pi)$ defined Lemma~\ref{lemma:bulk_closure}:
\begin{align}
    \Gamma_\textup{s}\left(m^{(\pi_2(1),\pi_2(2),\ldots,\pi_2(n)),(1,2,\ldots,n)}_{(1,\ldots,n)}(D_1D_2)\right)
    &= \left(\omega_1\omega_2, \sum_{i,j,k} \mathcal{B}_{\pi_2(k),i}(b_2) \mathcal{B}_{\pi_1(\pi_2(j)),\pi_2(k)}(b_1) r_i c_j \right) \\
    &= \left(\omega_1\omega_2, \sum_{i,j} \left( \sum_k \mathcal{B}_{j,k}(b_1) \mathcal{B}_{k,i}(b_2) \right) r_i c_{\pi_2^{-1}\circ \pi_1^{-1}(j))} \right) \\
    &= \left(\omega_1\omega_2, \sum_{i,j} \mathcal{B}_{j,k}(b) r_i c_{\pi^{-1}(j)} \right) \, .
\end{align}
Where $\pi = \pi_1\circ \pi_2$, i.e. $\pi^{-1} = \pi_2^{-1}\circ \pi_1^{-1}$, as in eq.~\eqref{eq.braidcomp}, and $b=b_1b_2$.
The last line follows from the braid group and multiplication of the Burau representation when concatenating braids.
\end{proof}

\begin{theorem}
If $b$ is a singular braid and $K_{b}$ is its long closure, then $\Gamma_\textup{s}(K_{b}) = (G(K_{b}),Z r_1 c_1)$.
\end{theorem}
\begin{proof}
If $b=(D,\pi)$ is a singular braid with $\Gamma_\textup{s}(D) = \left(1, \sum_{i,j} r_i A_{i,j} c_j \right)$, then its long closure is given by Lemma~\ref{lemma:bulk_closure}:
\begin{equation}
    \Gamma(K_{b}) = (\det(1-A_{ij}|_{i<n, 1<j}),Zr_1c_1) \, .
\end{equation}
Since $A_{i,j}$ is related to $\mathcal{B}_{ij}$ by taking the transpose and permuting the columns, and the determinant is invariant up to $\pm$ under a permutation of columns, and invariant under a transpose, this shows that $\Gamma_\textup{s}(K_{b}) = \left( \det_1(1-\mathcal{B}(b)), Z r_1 c_1\right) = \left(G(K_{b}), Z r_1 c_1\right)$.
\end{proof}

For singular braids, the invariant resulting from the $\Gamma_\textup{s}$-calculus is found to be equivalent to the invariant found by Gemein from the singular extension of the braid group, which he claims to be the singular Alexander polynomial. 
Thus, we name the scalar part of $\Gamma_\textup{s}$ the singular Alexander invariant and call it $\omega(K)$ for a (singular) long knot $K$ isotopic to $K_b$, given by:
\begin{equation}\label{eq:wk}
    \omega(K) = \det(1-\sum_{i,j} A_{i,j} |_{i<n,1<j} ) \, ,
\end{equation}
where $A$ is the given by $\Gamma_\textup{s}(b) = (1, \sum_{i,j} r_i A_{i,j} c_j)$ and $K_b$ is the long closure of $b=(D,\pi)$.

\subsection{Skein relations}
Ozsvath et al.~\cite{ozsvath2009floer} and Fiedler~\cite{fiedler2010jones} introduce a singular version of the Alexander polynomial, $\Delta^s(K)$, via a singular Skein relation.
$K\singular$ will denote a knot with at least a one singular crossing, $K^+$ ($K^-$) will denote the same knot where the singular crossing is replaced by a positive (negative) crossing.
Fiedler gives the Skein relation in terms of the Alexander polynomial $\Delta^s(K)$, which is assumed to have been found from the Kauffman state sum with weights such that $\Delta^s(K)\in\mathbb{Z}[A,A^{-1},B]$:
\begin{equation}\label{eq:fiedler}
    (A^{-1}-A) \Delta^s(K\singular) = (A^{-1}-B) \Delta^s(K^+) - (A-B) \Delta^s(K^-) \, .
\end{equation} 
The singular Skein relation used by Ozsvath et al. also presented in terms of the (generalized) Kauffman state sum $\Delta^s_{K}$ as a polynomial in $T,T^{-1}$:
\begin{equation}\label{eq:oszvath}
    (T^{-1/2}-T^{1/2}) \Delta(K\singular) = T^{-1/2} \Delta(K^+) - T^{1/2} \Delta(K^-) \, .
\end{equation}
Comparing the two Skein relations above (eq.~\eqref{eq:fiedler} and~\eqref{eq:oszvath}), we see that Fiedler's relation reduces to the Ozsvath et al. relation by setting $B=0$ and for $A=\sqrt{T}$ (which is the typical choice for $A$ to recover the Alexander polynomial from Fiedler's relation).

\begin{theorem}\label{thm:skein_gamma}
    the $\Gamma_\textup{s}$-calculus for a tangle $T$ satisfies the Skein relation:
    \begin{equation}\label{eq:skein_gamma}
    (t^{-1/2}-t^{1/2}) \Gamma_\textup{s}(T\singular) = (t^{-1/2} - s t^{1/2}) \Gamma_\textup{s}(T^+) - t^{1/2} (1 - s) \Gamma_\textup{s}(T^-) \, .
    \end{equation}
\end{theorem}
\begin{proof}
Consider a general 1-tangle $T$ with at least one singular crossing, and replace one singular crossing with a positive and with a negative crossing, see fig.~ \ref{fig:kauffman_K}. 
Before stitching the singular crossing to the remainder of the tangle consisting of $T_1$, $T_2$ and $T_3$, the $\Gamma_\textup{s}$ is given by:
\begin{align}
\Gamma_\textup{s}(T_1\sqcup T_2 \sqcup T_3\sqcup X\singular_{4,5}) &= 
    \left(\omega, \vec{r}A(T_1\sqcup T_2\sqcup T_3)\vec{c} + s \, r_4 c_4 + (1-s\,t) r_4 c_5 + (1-s)\, r_5 c_4 + s\, t \, r_5 c_5 \right)
    \label{eq:gamma_state_sum_sing}
    \\
    \Gamma_\textup{s}(T_1\sqcup T_2 \sqcup T_3\sqcup X_{4,5}^+) &= 
    \left(\omega, \vec{r}A(T_1\sqcup T_2\sqcup T_3)\vec{c} + r_4 c_4 + (1-t) r_4 c_5 + t \, r_5 c_5 \right)
    \label{eq:gamma_state_sum_pos}
    \\
    \Gamma_\textup{s}(T_1\sqcup T_2 \sqcup T_3\sqcup X_{5,4}^-) &= 
    \left(\omega, \vec{r}A(T_1\sqcup T_2\sqcup T_3)\vec{c} + r_5 c_5 + (1-t^{-1}) r_5 c_4 + t^{-1} \, r_4 c_4 \right)
    \label{eq:gamma_state_sum_neg}
\end{align}
with $\vec{r} = (r_1,r_2,r_3)$, $\vec{c} = (c_1,c_2,c_3)$. 
Performing the stitching 
$m^{1,3}_1 \circ m^{1,5}_1 \circ m^{1,2}_1 \circ m^{1,4}_1 $ 
to reconstruct the tangle gives the corresponding invariants of $T\singular$, $T^+$ and $T^-$.
Solving the Skein relation $\Gamma_\textup{s}(K\singular) = \alpha \Gamma_\textup{s}(K^+) + \beta \Gamma_\textup{s}(K^-)$ for the coefficients $\alpha$ and $\beta$ gives eq.~\eqref{eq:skein_gamma}
\end{proof}

\begin{figure}[tph!]
\captionsetup[subfigure]{justification=centering}
     \centering
     \hfill
     \begin{subfigure}[t]{0.25\textwidth}
         \centering
         \includegraphics[width=0.8\linewidth]{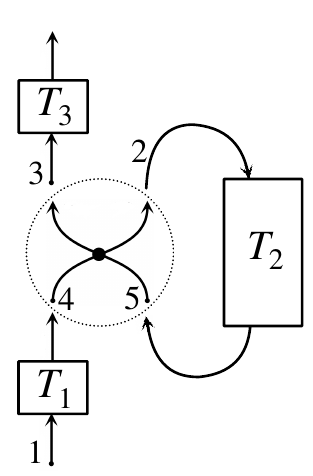}
         \caption{$T\singular$}
         \label{fig:singskein}
     \end{subfigure}
     \hfill
     \begin{subfigure}[t]{0.25\textwidth}
         \centering
         \includegraphics[width=0.8\linewidth]{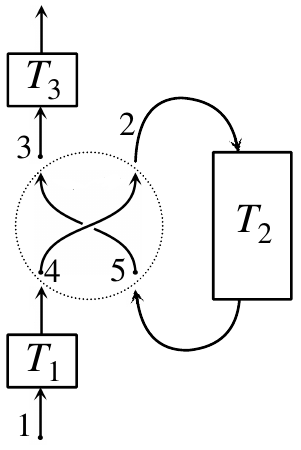}
         \caption{$T^+$}
         \label{fig:posskein}
     \end{subfigure}
     \hfill
     \begin{subfigure}[t]{0.25\textwidth}
         \centering
         \includegraphics[width=0.8\linewidth]{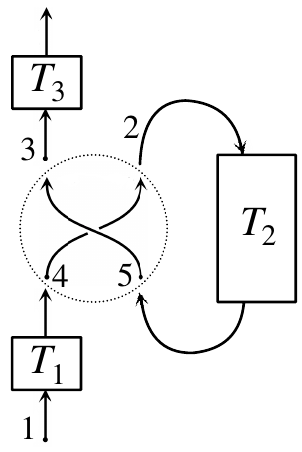}
         \caption{$T^-$}
         \label{fig:negskein}
     \hfill
     \end{subfigure}
        \caption{A singular tangle $T\singular$ (or knot $K\singular$) where one singular crossing is singled out and replaced by a positive crossing, giving $T^+$ (or $K^+$), or by a negative crossing, giving $T^-$ (or $K^-$). 
        The labelling is related to eqs.~\eqref{eq:gamma_state_sum_sing}-~\eqref{eq:gamma_state_sum_neg}.}
        \label{fig:kauffman_K}
\end{figure}


The scalar part of $\Gamma_\textup{s}(K)$ gives the invariant $\omega(K)$ corresponding to the long knot $K$, this invariant satisfies (by definition) the Skein relation of eq.~\eqref{eq:skein_gamma}.
We compare this Skein relation to Fiedler's relation in eq. \eqref{eq:fiedler} and find that for $A=\sqrt{t}$ and $B=s\sqrt{t}$, the invariants $\Delta^\textup{s}$ and $\omega$ satisfy the same relation. $A=\sqrt{t}$ is a standard choice here to recover the Alexander polynomial.

Similarly, we can see that $\Gamma_\textup{s}$'s Skein relation reduces to Ozsvath et al.~\cite{ozsvath2009floer} for $s=0$.
This is consistent with the relation between Fiedler $\Gamma_\textup{s}$: $B = s \sqrt{t}$).

\subsection{Relation to singular Kauffman state sum}\label{sec:kauffman}
Ozsvath et al. and Fiedler define their version of the singular Alexander polynomial based on a singular Kauffman state sum. Here we extend the Kauffman state sum to singular knots in the $\Gamma_\textup{s}$-calculus.

The Kauffman state sum provides a polynomial $\langle{K}\rangle$ defined as \cite{kauffman2006formal}:
\begin{equation}\label{eq:state_sum}
    \langle{K}\rangle= \sum_{S\in U} \langle K | S \rangle \, ,
\end{equation}
where $S$ is a state within the set of states $U$, and $\langle K | S \rangle$ is the inner product between the knot diagram $K$ and the state $S$, $\langle K | S \rangle = \sigma(S)w_1(S) \ldots w_n(S)$.
Here, $\sigma(S)=(-1)^B$ is the sign of the state $S$ determined by the number of black holes, $B$, where a black hole is defined to be the marker `beneath' a crossing; see fig.~ \ref{fig:kss_proof}.
Furthermore, $w_i(S)$ is the weight of the $i^\text{th}$ crossing in state $S$.
A state is a diagram of a long knot with one marker at every crossing, such that there is one marker in every inner region.
This is illustrated in fig.~ \ref{fig:state_sum_example} for the trefoil knot\index{knot!trefoil}. 
For an in-depth revision of the Kauffman state sum we refer to ref.~\cite{kauffman2006formal}.
The state sum polynomial is the determinant of the reduced Alexander matrix, $\langle K \rangle = \det(A)$~\cite{kauffman2006formal}.

\begin{figure}[tph!]
    \centering
    \includegraphics[width=0.75\textwidth]{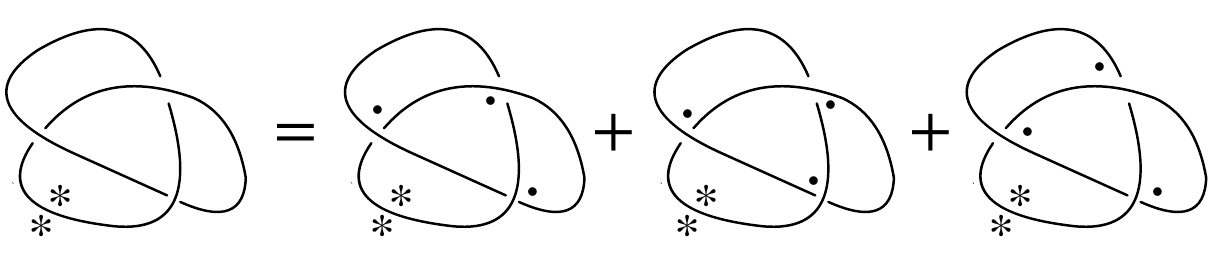}
    \caption{Example of the set of states of the trefoil knot. 
    To find the long knot, one region is `cut open', denoted by the star. 
    The three states are found by placing one weight at every crossing in every region; the weights are denoted by a dot to show the distinctness of the the three states.
    The value of the weights for the singular state sum are shown in fig.~\ref{fig:kss_proof}.}
    \label{fig:state_sum_example}
\end{figure}

\begin{proposition}\label{prop:sing_kauffman}
    Choosing the right assignments of regions and vertices for singular knots, the Kauffman state sum is the determinant of the generalised singular Alexander matrix with rows corresponding to vertices and columns corresponding to regions. The entries of the matrix are given by the weights $w$.
    \begin{equation}
    \Delta^\textup{s}(K) = \det(A) = \sum_{\sigma\in S_n} \text{sgn}(\sigma) a_{1\sigma(1)} \cdots a_{n\sigma(n)} = \langle K \rangle = \sum_{S\in U} \sigma(S) \langle K | S \rangle \, .
    \end{equation}
\end{proposition}
Here we have extended Kauffman's original Kaufmann state sum to include singular crossings by introducing weights for the singular crossing.
\begin{figure}[tph!]
    \centering
    \includegraphics[width=0.5\textwidth]{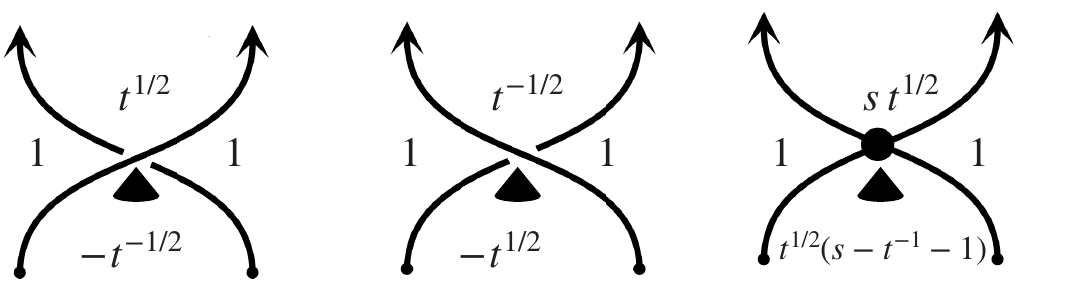}
    \caption{The weights are shown for each of the four regions that neighbour a crossing, for each of the three crossing types.
    The `black hole', which determines the sign of the state in eq.~\eqref{eq:state_sum} is indicated by the cone-like symbol.
    }
    \label{fig:kss_proof}
\end{figure}
%


%
The weights for the positive and negative crossing can be derived from the fundamental group via the Fox calculus on the Dehn presentation.
However, to our knowledge, this is not possible for the singular weights.
Therefore, we assign weights to the regions based on the $\Gamma_s$-calculus. 
The resulting weights are depicted in fig.~\ref{fig:kss_proof}.
%
%
By setting $s=1$, the singular crossing weights equal the positive crossing weights.
The weights in fig.~\ref{fig:kss_proof} allow for the construction of a singular variant of the Alexander matrix. 

\begin{proposition}
    The singular Alexander polynomial constructed from the singular Kauffman state sum using the weights in fig.~\ref{fig:kss_proof} for a singular knot $K$ is equal to $\omega(K)$, i.e. the invariant resulting from $\Gamma_\textup{s}(K)$.
\end{proposition}
\begin{proof}
    One can see that the singular variant of the Kauffman state sum satisfies the singular Skein relation for $\Gamma_\textup{s}$, i.e. eq.~\eqref{eq:skein_gamma}.
    Thus, the invariants are equal upto $\pm t^{k}$, $k\in\mathbb{Z}$.
    The Kauffman state sum is given as the determinant of the singular Alexander matrix, and has shown in sec.~\ref{subsec:gemein}, this is also how the invariant $\omega(K)$ is found, see eq.~\eqref{eq:wk}.
    Additionally, these weights match the weights in Feidler's singular state sum for $B=s\sqrt{t}$ and $A=\sqrt{t}$, the same equality matching found in the previous section.
    Furthermore, they match the weights of used by Oszvath at al. to find the singular Alexander polynomial for $s=0$.
\end{proof}

\section{Perturbed singular Alexander}\label{sec:perturbed}
We aim to extend the $\Gamma_\textup{s}$-calculus in order to construct a better invariant, in a similar way as was done in refs. \cite{bar2021perturbed,bar2022perturbed} for the $\Gamma$-calculus, where the invariant $(\Delta,\rho_1)(t)$ was introduced.
We introduce the singular version which we name $(\Delta^\textup{s},\rho^\textup{s}_1)(t,s)$ with $\Delta^\textup{s}$ being the two-variable singular Alexander polynomial and $\rho^\textup{s}_1$ the two-variable singular variant of $\rho_1$, which can be seen as a perturbation to the Alexander polynomial.
The invariant $\rho_1$ is stronger then $\Delta(t)$, while still easy to compute.
Similarly, the invariant $\rho^\textup{s}_1$ is still easy to compute and is stronger than the singular Alexander polynomial (introduced in the previous section as the scalar part of $\Gamma_\textup{s}$).
We will show this with an example in sec.~\ref{sec:rho1_example}.

\subsection{Rotational knot diagrams}\label{sec:rot_c}
\begin{wrapfigure}{r}{0.18\textwidth}
    \centering
    \includegraphics[width=0.18\textwidth]{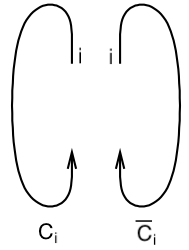}
    \caption{Curl diagrams.}
    \label{fig:curl}
\end{wrapfigure}
To improve the invariant with respect to the previous section, we introduce a way to keep track of the rotation number of the strands of the diagram. 
We define the rotational tangle diagrams following ref.~\cite[p.~31]{bar2021perturbed}, and introduce the crossingless strands $C_i$ and $\overline{C}_i$ that rotate counter-clockwise and clockwise, respectively, and we call them \textit{curls}; see fig.~\ref{fig:curl}.

Following again the definition in ref.~\cite[p.~32]{bar2021perturbed}, the disjoint union of two rotational tangle diagrams is the disjoint union as defined in def.~\ref{def:disjoint-union}.
Merging as presented in def.~\ref{def:merging} has the additional constraint that the arc that connects the two strands has rotation number $0$.

\subsection{Perturbed singular invariant}\label{sec:intro_perturbed}
We introduce recursive relations that indicate how to construct the matrix $G$.
\begin{lemma}[g-rules]\label{lem:grules}
For the positive and negative crossings (with overstrand $i$ and understrand $j$) the recursive relations that determine the $ij$-part of the matrix $G$ are:
\begin{align}
    g_{i,\beta}= g_{i^+,\beta }+\delta _{i,\beta } \, , \qq{}
    g_{j,\beta}= \left(1-t^{\pm}\right) g_{i^+,\beta }+t^{\pm} g_{j^+,\beta }+\delta _{j,\beta } \label{eq:grules}
\end{align}
\end{lemma}
\noindent Due to a difference in conventions these are not the same as in ref. \cite{bar2022perturbed}, where $\rho_1$ was introduced.

\begin{lemma}[Singular g-rules]\label{lem:sgrules}
For the singular crossing with the left strand labelled $i$ and the right strand labelled $j$ the recursive relation that determines the $ij$-part of the matrix $G$ is:
\begin{align}
    g_{i,\beta} =  s \, g_{i^+,\beta} + (1-s) \,  g_{j^+,\beta} + \delta_{i,\beta} \, , \qq{}
    g_{j,\beta} = (1- s t) \, g_{i^+,\beta} + s t \, g_{j^+,\beta }+\delta _{j,\beta } \label{eq:sgrules}
\end{align}
\end{lemma}
\noindent In the case $s=1$, the singular crossing relation reduces exactly to the positive crossing relation.

\begin{proof}
    A sketch of the derivation of the (singular) g-rules can be found in sec.~\ref{app:Heisenberg}.
\end{proof}

The matrix $G$ that is constructed form the g-rules is the inverse of the the matrix $A$, which has the property that $\det(1+A^T) = \det(1+A)=\Delta_K^s$.
Thus, the g-rules can be used to find the Alexander polynomial.

\begin{definition}\label{def:notperturbation}
We define the total matrix $A$ of a knot the sum over the crossings, $A=\sum_c A_c$, with $A_c$ a $(2n+1)\cross(2n+1)$ size matrix with zeros except for the block corresponding to the crossing $c$, where it is given as:
\begin{align}
    A^\pm = 
    \left[\begin{array}{c|c c} 
	1 & i^+ & j^+ \\ 
    \hline
    i & -1 & 0 \\ j & -1+t^\pm  & -t^\pm
    \end{array}\right]
    \, , \qq{}
    A\singular = 
    \left[\begin{array}{c|c c} 
    1 & i^+ & j^+ \\
    \hline
    i & - s & -1+s \\ j & -1+st & -s t
    \end{array}\right]
    \, , \label{eq:A_matrix}
\end{align}
with the superscript indicating the type of crossing. 
\end{definition}

The blocks in the matrix $A_c$ given in eq.~\eqref{eq:A_matrix} are equal to the matrix part of $- \Gamma_\textup{s}^T$ in eqs. \eqref{eq.regularX},~\eqref{eq.singularX}.
Thus, the matrix part of $\Gamma_\textup{s}^T$ is the minus of the transpose of the Alexander matrix.
The the Alexander polynomial is given by $\det(1+A) = \det(1-\Gamma_\textup{s}^T) = \det(1-\Gamma_\textup{s})$ (since the determinant is invariant under the transpose). 
We can further normalise to determine:
\begin{equation}
    \Delta^\textup{s}_D = t^{(\phi(D)-w(D)-x(D))/2}\det(1+A)\, ,
\end{equation}
with $A$ build as described above, $\phi$ the total rotation number of the diagram $D$ (the sum of the signs of the curls), $w$ the wright (the total sum of the signs of the crossings, a singular crossing having sign zero), and $x$ the number of singular crossings of the diagram $D$.

Additionally we now define $\rho_1^\textup{s}$ , the second part of the invariant pair $(\Delta^\textup{s}, \rho_1^\textup{s})$.
By defining the crossings at quadratic order in terms of $g_{i,j}$, we can find a stronger invariant from the g-rules.
Since this is like a second-order perturbation effect (in $g$), we refer to it as the \textit{perturbed} Alexander polynomial.
This invariant has it's background in quantum algebra, as explained in ref.~\cite{bar2021perturbed}, which we touch upon in sec.~\ref{app:Heisenberg}.
We first define for each type of crossings an $R$-matrix given in terms of the two-point g-functions at quadratic order, and the curls $C$ we express in terms of two-point g-functions at linear order.
\begin{definition}[Crossing $R$-matrices]\label{def:gcros}
We denote the positive and negative crossings with overstrand $i$ and understand $j$ as $R_1[+,i,j]$, $R_1[-,,i,j]$, the singular crossing with the left-most beginpoint labelled $i$ and the right-most $j$ as $sR[i,j]$, and the counter-clockwise/clockwise curl of strand $i$ as $C[\pm,i]$.
    \begin{align}
    R_1[\pm,i,j] &= \mp \left(t^{\pm 2}+t^{\pm 1}-2\right) g_{i,j}^2 \mp g_{i,j} \left(g_{j,i}+2 g_{j,j}+t^{\pm 1}\right) \pm g_{i,i} \left([t^{\pm 1}+3] g_{i,j}-g_{j,j}+1\right) \mp \frac{1}{2}
    \\ \label{eq:R1}
    C[\pm,i] &= \pm g_{i,i} \mp \frac{1}{2} \\
    sR_1[i,j] &= 
    c_1 g_{i,i} g_{i,j} 
    - c_2 g_{i,i} g_{j,i} 
    + c_3 (g_{i,i} g_{j,j} + g_{i,j} g_{j,i})
    + c_4 g_{i,j}^2
    + c_5 g_{j,j}^2
    - c_6 g_{i,j} g_{j,j} \nonumber \\
    &\qq{} + c_7 g_{j,i}^2 - c_8 g_{j,i} g_{j,j} + s (g_{i,i} - t \, g_{i,j}) - \frac{1}{2}
    \label{eq:R3}
    \end{align}
\end{definition}
For the curl $C$, the plus sign corresponding to a counter-clockwise rotation and the minus sign corresponding to a clockwise rotation (denoted $C_i$ and $\overline{C}_i$ in fig.~\ref{fig:curl}, respectively).
The coefficients $c_i$ are:
\begin{align}
    c_1 &= \frac{s\,t}{(t - 1) (s - 1 + s t)^2}(-1 - 3 t + s^2 t (-1 + t^2) + s (1 + t + 2 t^2))\, , \nonumber \\
    c_2 &= \frac{2}{(t - 1) (s - 1 + s t)^2} (-1 + s) s (1 - s (1 + t)^2 + s^2 t (1 + t)^2)\, , \nonumber \\
    c_3 &= \frac{s\,t}{(t - 1) (s - 1 + s t)^2} (2 t + 3 s (1 + t) + 2 s^3 t (1 + t)^2 - s^2 (3 + 6 t + 5 t^2 + 2 t^3)) \, , \nonumber \\ 
    c_4 &= \frac{s\,t^2}{(t - 1) (s - 1 + s t)^2} (-1 + s (-1 + t) + t + s^2 (t - t^3))\, , \nonumber \\
   c_5 &= \frac{s\,t}{(t - 1) (s - 1 + s t)^2}(s-1) (-1 + t^2) (-1 - s + s^2 t (1 + t)) \, , \nonumber \\
   c_6 &= \frac{2 s\,t}{(t - 1) (s - 1 + s t)^2} (-1 - t + t^2 + s^3 t^2 (1 + t)^2 + s (1 + 2 t + 2 t^2) - 
   s^2 t (2 + 3 t + 2 t^2 + t^3))\, , \nonumber \\
   c_7 &= \frac{s}{(t - 1) (s - 1 + s t)^2}(s-1) (-1 + t) (-1 + s^2 (1 + t)^2)\, , \nonumber \\
   c_8 &= \frac{1}{(t - 1) (s - 1 + s t)^2}((s-1) s (-1 + s - t - 2 t^2 - s t^2 + 
     2 s^2 (-1 + t) t (1 + t)^2)) \, . \nonumber
\end{align}

There is some freedom in defining the expressions in Definition \ref{def:gcros}. 
We chose the definition for the positive/negative crossing and curl such that it matches the invariant $\rho_1$ discussed in ref.~\cite{bar2022perturbed}.
The singular $R$-matrix was chosen such that for $s=1$, it reduces to the positive crossing $R$-matrix.
As a result, if $s=1$, the singular invariant $\rho_1^\textup{s}$ matches the non-singular invariant $\rho_1$ in ref.~\cite{bar2022perturbed} for the knot where the singular crossings are replaced by positive crossings.
Other choice may result in less complex expressions for $sR_1$.
We now define the perturbed invariant as the sum over all crossings and curls in the diagram.

\begin{definition}[$\rho_1^\textup{s}$]\label{def:rho1}
For a knot $K$ represented by the diagram $D$, the invariant contribution from the perturbation is:
\begin{equation}\label{eq:rho1s}
    \rho_1^\textup{s}(K) = -\left(\Delta^\textup{s}_K\right)^2 \left( \sum_{c\in D} R_1(c) + \sum_k \phi_k \left[ g_{kk} - \frac{1}{2}\right] \right) \, ,
\end{equation}
with the sum over the crossings $c$ over all the $R$-matrices of the crossings, and the sum over the edges $k$ over the rotation number $\phi$, and with the crossings as in Definition \ref{def:gcros}.
\end{definition}

The factor minus $\Delta^\textup{s}_K$ squared is used as a normalisation to match with previous results; the non-singular knots in the table in Appendix~\ref{app:knot_table} match the $\rho_1$ in ref.~\cite{bar2022perturbed}.

\begin{proof}[Proof: Invariance of $\rho_1^\textup{s}$]
    The g-rules in Lemmas \ref{lem:grules} and \ref{lem:sgrules} can be interpreted as assigning a weight to each strand that goes out of a crossing.
    For example, the ingoing strand $j$ assigns a weight $(1-t)$ to the outgoing stand $i^+$ and a weight $t$ to the outgoing stand $j^+$ according to the g-rule for $g_{i,\beta}$.
    Interpreting the g-rules this way, and multiplying the weights gained through each crossing, the weights at the end of each strand are invariant under the Reidemeister moves, see for example fig.~ \ref{fig:example-sr3}.

    The expressions in def.~\ref{def:gcros} were found by first expressing the eqs.\eqref{eq:R1}-\eqref{eq:R3} (the $R$-matrices and curl definitions) in terms of arbitrary coefficients, and then solving the conditions from the Reidemeister moves depicted in fig.~\ref{fig-knots:RMmoves} and additionally fig.~\ref{fig-knots:sRM}, for the coefficients, using the g-rules.
    This derivation can be found on Github, see the link in footnote~\ref{footnote:github-knots}.
    We have solved the expressions for the crossings and curl explicitly such that the weight at the end (i.e. $\rho_1^\textup{s}$) is invariant under the Reidemeister moves. Therefore, Defenition~\ref{def:rho1} provides an invariant.
    
\end{proof}


\begin{figure}[thb]
     \centering
     \begin{subfigure}[t]{0.49\textwidth}
         \centering
         \includegraphics[width=0.7\linewidth]{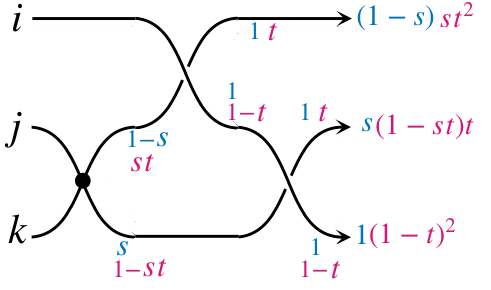}
     \end{subfigure}
     \hfill
     \begin{subfigure}[t]{0.49\textwidth}
         \centering
         \includegraphics[width=0.7\linewidth]{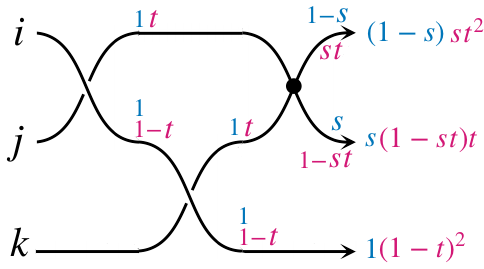}
     \end{subfigure}
        \caption{Example of the interpretation of the (singular) g-rules in eqs. \eqref{eq:grules}-\eqref{eq:sgrules} as the distribution of weights on the strands, invariant under the third singular Reidemeister move.}
        \label{fig:example-sr3}
\end{figure}


\begin{figure}[tbh]
     \centering
     \begin{subfigure}[b]{0.49\textwidth}
         \centering
         \includegraphics[width=0.85\linewidth]{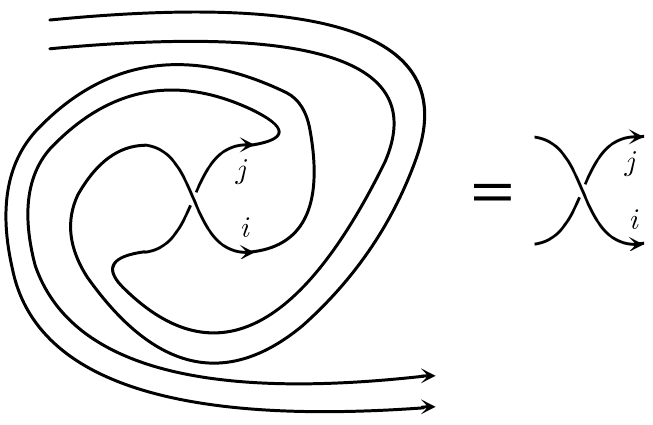}
     \end{subfigure}
     \hfill
     \begin{subfigure}[b]{0.49\textwidth}
         \centering
         \includegraphics[width=0.85\linewidth]{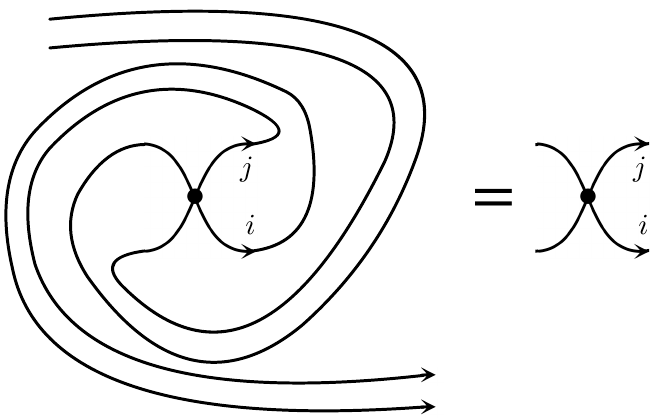}
     \end{subfigure}
    \caption{Additional singular Reidemeister moves used to determine the coefficients in Definition~\ref{def:rho1}.}
    \label{fig-knots:sRM}
\end{figure}

\subsection{Examples of improvement}\label{sec:rho1_example}
As shown in the knot table in appendix \ref{app:knot_table}, the singular Alexander polynomial found from $\Gamma_\textup{s}$ already distinguishes more knot classes then the non-singular Alexander polynomial.
For example, the extra information embedded in the singular crossing helps distinguish the right-handed singular trefoil and left-handed singular trefoil\index{knot!trefoil} (singular trefoil meaning the trefoil knot where at least one crossing is singular).
The knot table also shows that the singular Alexander polynomial from $\Gamma_\textup{s}$ and from the g-rules match, as was shown in sec.~\ref{sec:intro_perturbed}.
For the knots included in the knot table, which is knots with up to five crossings, the presence of a singular crossing already distinguishes the left-and right-handed variants that the non-singular $\Gamma$ does not distinguish. 
Reference~\cite{bar2022perturbed} showed that the perturbed Alexander invariant \textit{does} distinguish these left-and right-handed variants, which was our motivation for introducing a singular perturbed Alexander invariant.
In example~\ref{ex:improvement} we show how $\Delta_K^s$ does not distinguish two knot classes, but $\rho_1^\textup{s}$ does.

In general, the Alexander polynomials of two knots that differ by a double-delta move (depicted in the circle in fig.~\ref{fig:double_delta}) are indistinguishable.
Since adding a singular crossing is the same on the right-and left-hand side of the double delta move, the two sides of the double delta move with a singular crossing attached at the same place are also indistinguishable when the tangle is closed to form a long knot. 
We would expect the Alexander polynomial to not distinguish such types of singular knots, while the perturbed Alexander polynomial does, i.e. it distinguishes $S$-equivalent knots~\cite{naik2003move}. 


\begin{example}\label{ex:improvement}
    We consider two singular knots related by a double-delta move and show that although the singular Alexander polynomial from $\Gamma_\textup{s}$ is the same for both tangles, the perturbed Alexander polynomial $\rho_1^\textup{s}$ is able to distinguish them and therefore is an improved invariant.
    We consider the simplest closures of the double delta moves shown in fig.~\ref{fig:double_delta}.
    From $\Gamma_\textup{s}$ we find the scalar part to be:
    \begin{align}
        \Gamma_\textup{s}(K_l) &= -\frac{2 s^2}{t}+4 s^2+\frac{s}{t}-3 s+\frac{1}{t}\, ,\\
        \Gamma_\textup{s}(K_r) &= -\frac{2 s^2}{t}+4 s^2+\frac{s}{t}-3 s+\frac{1}{t} \, .
    \end{align}
    From the perturbed invariant we find:
    \begin{align}
        (\Delta^\textup{s},\rho_1^\textup{s})(K_l) &= \Bigg(\frac{4 s^2 t-2 s^2-3 s t+s+1}{t},-\frac{(s-1)}{(t-1) t^2 (s t+s-1)^2} \left[5 s^5 t^7-13 s^5 t^6-12 s^5 t^5 \right.  \nonumber \\ 
        &\qq{} + 26 s^5 t^4+10 s^5 t^3 
        -10 s^5 t^2+s^5 t+s^5+2 s^4 t^6-5 s^4 t^5+59 s^4 t^4-81 s^4 t^3 \nonumber \\ 
        &\qq{}+5 s^4 t^2+6 s^4 t-2 s^4-4 s^3 t^5-17 s^3 t^4 -5 s^3 t^3+53 s^3 t^2-10 s^3 t+s^3 -2 s^2 t^4 \nonumber \\ 
        &\qq{} \left.+38 s^2 t^3-44 s^2 t^2-s^2 t-s^2-s t^3-4 s t^2+3 s t+2 s+t-1\right]\Bigg)
    \end{align}
    and
    \begin{align}
        (\Delta^\textup{s},\rho_1^\textup{s})(K_r) &= \Bigg( \frac{4 s^2 t-2 s^2-3 s t+s+1}{t},\frac{1}{(t-1) t^2 (s t+s-1)^2}\left[ 3 s^6 t^7+5 s^6 t^6-2 s^6 t^5 \right. \nonumber \\ 
        &\qq{}-10 s^6 t^4 -6 s^6 t^3+2 s^6 t^2 +s^6 t-s^6-13 s^5 t^7-12 s^5 t^6+56 s^5 t^5-75 s^5 t^4 \nonumber \\ 
        &\qq{} +55 s^5 t^3 +24 s^5 t^2-14 s^5 t+3 s^5+9 s^4 t^7+35 s^4 t^6 -86 s^4 t^5+70 s^4 t^4 \nonumber \\ 
        &\qq{}+25 s^4 t^3-111 s^4 t^2 +27 s^4 t-3 s^4-24 s^3 t^6+5 s^3 t^5+66 s^3 t^4-136 s^3 t^3 \nonumber \\ 
        &\qq{}+124 s^3 t^2 -9 s^3 t+2 s^3 +22 s^2 t^5-41 s^2 t^4+45 s^2 t^3-23 s^2 t^2-10 s^2 t-3 s^2 \nonumber \\ 
        &\qq{}\left.-8 s t^4 +16 s t^3-14 s t^2+3 s t +3 s+t^3-2 t^2+2 t-1 \right] \Bigg) \, .
    \end{align}
    Although the singular Alexander polynomials are the same, the perturbed singular Alexander, $\rho_1^\textup{s}$ distinguishes the two knots.

\begin{figure}[thb]
    \captionsetup[subfigure]{justification=centering}
     \centering
     \begin{subfigure}[t]{0.46\textwidth}
         \centering
         \includegraphics[width=\linewidth]{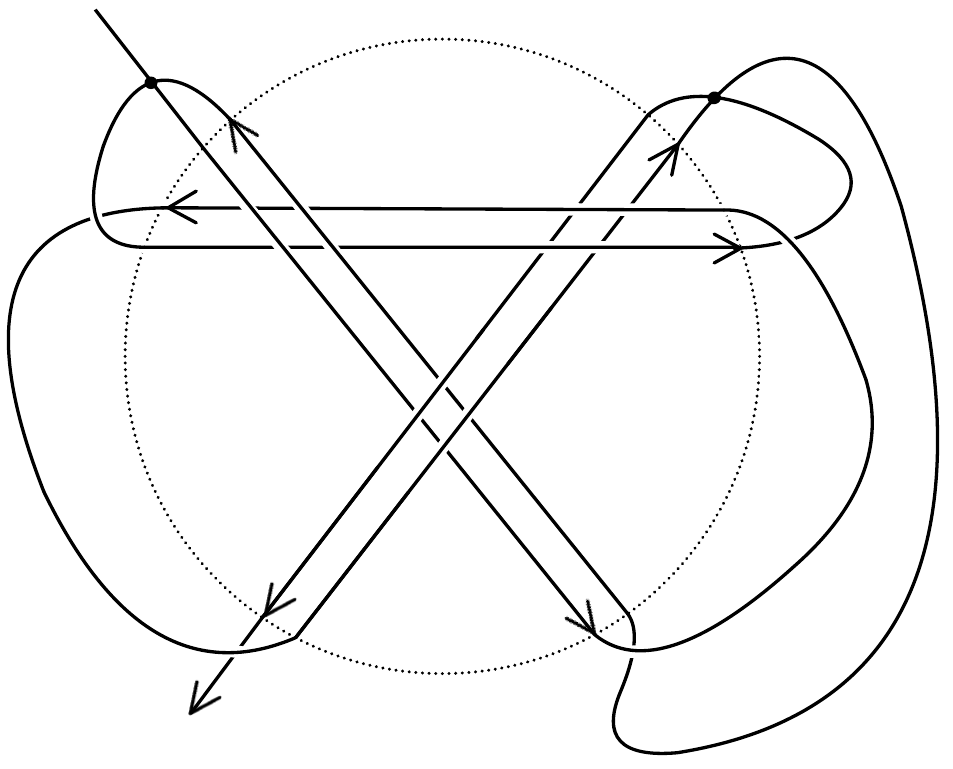}
         \caption{$K_l$}
         \label{fig:double_delta-lhs}
     \end{subfigure}
     \hfill
     \begin{subfigure}[t]{0.46\textwidth}
         \centering
         \includegraphics[width=\linewidth]{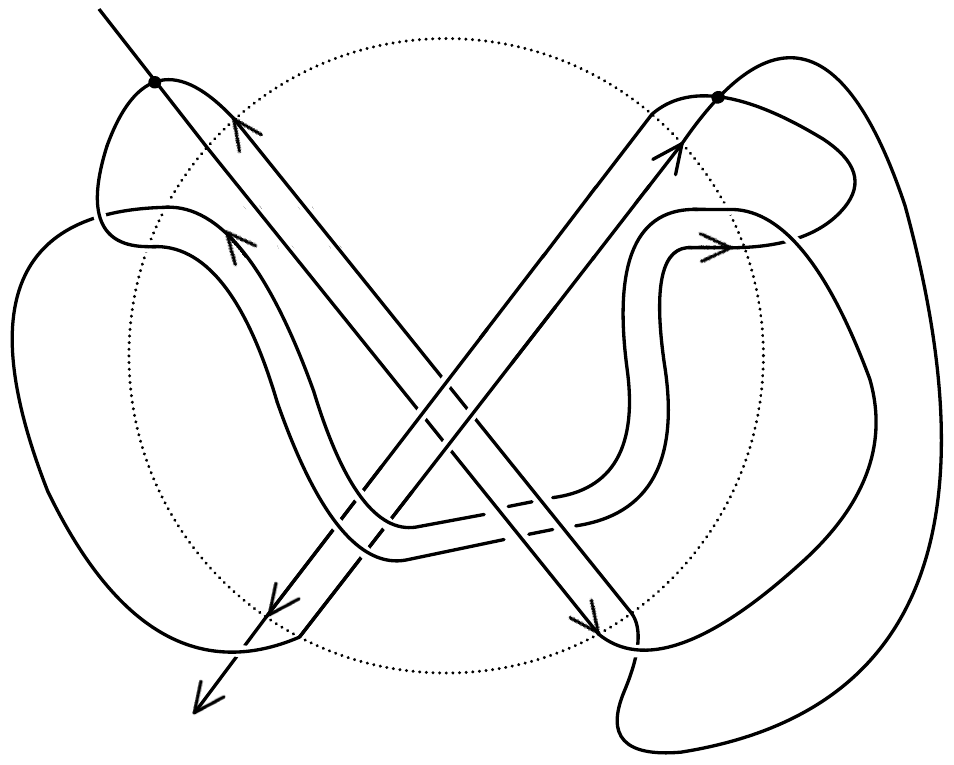}
         \caption{$K_r$}
         \label{fig:tdouble_delta-rhs}
     \end{subfigure}
        \caption{Double delta move with a singular crossing attached.}
        \label{fig:double_delta}
\end{figure}    
\end{example}

Example~\ref{ex:improvement} illustrates that $\rho_1^\textup{s}$ is an improved invariant compared to $\Delta^\textup{s}$. 
However, it still has some short-comings; for example, it does not distinguish the Conway- and Kinoshita-Terasaka knots, which are related by a mutation, or the singular Conway- and Kinoshita-Terasaka knots where a standard crossing outside of the mutation area is replaced by a singular crossing. 
This is shown in the code on Github, see footnote~\ref{footnote:github-knots}.


\section{Summary}
We have provided a fast and computationally easy way to compute the singular Alexander polynomial for singular knots.
Although the crossing relations for the non-singular crossings can be derived from the fundamental group, this becomes problematic for the singular crossing, this was left an open question in this thesis.
The crossing relations were instead found by solving the singular Reidemeister moves.

We have also shown an improved invariant for singular knots: $\rho_1^\textup{s}$, which distinguishes a class of knots that are $S$-equivalent.
Although the invariant property of $\rho_1^\textup{s}$ may not be immediately evident, this becomes more apparent in the context of quantum algebras, where the invariant had its origin.

The knot invariants introduced in this chapter are generally easy to find computationally (with the code that has been made available on GitHub, see footnote~\ref{footnote:github-knots}).
Even for knots with a larger number of crossings this does not take long or much computational power. 
The singular knot table in appendix~\ref{app:knot_table}) shows the invariants for knots up to $5$ crossings that have been found this way, and may provide a starting point for creating a larger knot table for singular knots.

\bibliographystyle{ieeetr}
\bibliography{biblio.bib}

\newpage
\appendix

\section{Background of perturbed Alexander invariant}\label{app:Heisenberg}
In this section we give some background regarding the development of the perturbed Alexander polynomial, which is based on ref.~\cite{bar2021perturbed} where the authors proposed a technique for turning algebra into perturbed Gaussians, and found an associated universal quantum invariant for knots.
Expanding this universal invariant, $Z_\epsilon(K)$, in terms of $\epsilon$, at zeroth order it is equivalent to the Alexander polynomial and at first order it is equivalent to $\rho_1$ (which we refer to as the perturbed Alexander polynomial).
At first order, the expansion of $Z_\epsilon(K)$ (i.e. $\rho_1$) is found from the $R$-matrices (which is why in Definition~\ref{def:gcros} the $R$'s had subscript $1$).
The $R$-matrices are quartic polynomials in the elements that generate the algebra, but can alternatively be computed using the $g$-rules, which were introduced in ref.~\cite{bar2022perturbed} and in sec.~\ref{sec:perturbed} for singular knots.
We briefly introduce the relevant algebra and how the $g$-rules can be used to compute $\rho_1$; for a full explanation we refer to refs.~\cite{bar2021perturbed,bar2022perturbed}.

Starting from the Heisenberg algebra $H$, generated by $p,x$ which have the relation $[p,x]=1$, we can compute the algebra by doing all the commutations.
Alternatively, we can propose a generating function, which happens to take the shape of a Gaussian function:
\begin{equation}
    e^{(\pi_i + \pi_j)p}e^{(\xi_i + \xi_j)x}e^{-\xi_i\pi_j}\, .
\end{equation}
The generating function normal-orders (alphabetically) the generating elements $p$ and $x$, which have coefficients $\pi_{i}$, $\pi_j$ and $\xi_i$, $\xi_j$, respectively; the indices $i$ and $j$ corresponding to strand-labels.
To find a knot invariant, we associate an $R$-matrix to every crossings; where the $R$-matrix is given in exponential form as well, i.e. $R=e^{p\otimes x - 1 \otimes px}$. 
The first term in the exponent is on the overstrand, and the second term in the exponent is on the understrand. 
Using Gaussian functions is advantageous for computing a knot invariant when considering the multiplication of many $R$-matrices that construct a knot.
The knot invariant is computed by multiplying all the terms in order of appearance along the knot. 
The generating function for multiplication is:
\begin{equation}
    m^{ij}_k = e^{(\pi_i + \pi_j)p_k + (\xi_i + \xi_j)x_k - \xi_i\pi_j}\, ,
\end{equation}
and the crossing $R$-matrices are:
\begin{equation}
    R_{ij}^\pm = e^{\left(T^\mp-1\right)\left(p_i-p_j\right)x_j} \, , \qq{} R_{ij}\singular = e^{\left(s-1+s T\right)^{-1}\left(p_i-p_j\right)\left[(s-1)x_i + (1- sT)x_j\right]} \, .
\end{equation}
These are solutions to the (singular) Reidemeister moves, see footnote~\ref{footnote:github-knots}. 
The computation of the invariant from multiplication along the knot will relate to the Alexander polynomial~\cite{bar2021perturbed}.
In order to improve the invariant we add a perturbation to the Gaussians:
\begin{equation}
    \mathcal{R}_{ij}^{\pm,\sbullet[0.8]} = R_{ij}^{\pm,\sbullet[0.8]} \left( 1 + \epsilon P(p_i,p_j,x_i,x_j) + \order{\epsilon^2} \right) \, ,
\end{equation}
with $P$ a polynomial in $p_i,p_j,x_i,x_j$ at quartic order.
For each crossing, this perturbation adds normal-ordered $p$'s and $x$'s at the starting point of the strands.
When finding the invariant, we push these perturbations through the crossings, and takes the multiplication of all constant contributions from the commutations (such as illustrated in fig.~\ref{fig:example-sr3} for the found weights).
This pushing can be done using the relations:
\begin{align}
    (p\otimes 1) R_{ij} &= R_{ij} (p\otimes 1) + \partial_{x\otimes 1} R_{ij} \qq{} \text{and} \qq{}
    (1\otimes p) R_{ij} = R_{ij} (1\otimes p) + \partial_{x\otimes 1} R_{ij} \, .
\end{align}
For the different crossings and after some manipulation, this becomes:
\begin{align}
    p_i R_{ij}^\pm &= R_{ij}^\pm p_i \, , \hspace{2mm}
    &&p_j R_{ij}^\pm = R_{ij}^\pm \left( T^\pm p_j - (T^\pm-1) p_i \right) \, , \hfill \\
    p_i R_{ij}\singular &= R_{ij}\singular \left( s\,p_i - (s-1) p_j\right)\, ,
    &&p_j R_{ij}\singular = R_{ij}\singular (s\,T \,p_j + (1-s\,T)p_i)  \, .
\end{align}
These equations are reminiscent of the rows of the the Alexander matrices given in Definition~\ref{def:notperturbation}. 
They describe how a $p$ element is pushed from the start of the crossing (when it is before the exponent), through the exponent that is on the crossing, towards the end of the crossing (when it is after the exponent).
Alternatively we can use a two-point function $g_{k,\beta}$ to describe how a $p$ element on strand $k$ is pushed through the crossing, when there is an $x$ element on strand $\beta$ with which it has the commutation relations $[p,x]=1$. 
For the positive and negative crossings these two-point functions are:
\begin{align}
    g_{i,\beta}= g_{i^+,\beta }+\delta _{i,\beta } \, , \qq{}
    g_{j,\beta}= \left(1-t^{\pm}\right) g_{i^+,\beta }+t^{\pm} g_{j^+,\beta }+\delta _{j,\beta }  \, ,
\end{align}
and for the singular crossings:
\begin{align}
    g_{i,\beta} =  s \, g_{i^+,\beta} + (1-s) \,  g_{j^+,\beta} + \delta_{i,\beta} \, , \qq{}
    g_{j,\beta} = (1- s t) \, g_{i^+,\beta} + s t \, g_{j^+,\beta }+\delta _{j,\beta }  \, . 
\end{align}
These are the g-rules given in Lemma~\ref{lem:grules} and~\ref{lem:sgrules} that were used to find an invariant expression in terms of g-functions for the crossings.
These crossing relations could also have been found by contracting the $p$ and $x$ pairs in the polynomials of the perturbation of $\mathcal{R}$.

\section{Singular knot table}\label{app:knot_table}
give an (incomplete) overview of some singular knots up to $5$ crossings to illustrate that the $\Gamma_\textup{s}$-calculus can make some distinctions that the $\Gamma$-calculus cannot because of the addition of a singular crossing. We also present the polynomials from the perturbed Alexander polynomial $\rho^\textup{s}_1$. 

The knots chosen in table~\ref{table:knot_table}, which have one singular crossing, align with the knots chosen in ref. \cite{gemein2001representations}, although they didn't take handedness into account explicitly. 
Also, the chosen knots coincide with ref. \cite{oyamaguchi2015enumeration}, which gave a more complete overview of all possible knots up to 6 crossings, where one crossing is singular.

We note that up to $5$ crossings $\Gamma_\textup{s}$ can distinguish left and right-handed singular knots, i.e. for $3_1, 4_1, 5_1, 5_2$, if the same crossing is singular for the left- and right-handed knots, then their singular Alexander polynomials are different. 
However, if we were to set $s=1$, the singular Alexander polynomials become indistinguishable. 
The extra information provided by the singular crossing - represented by the extra variable $s$ - distinguishes the left- and right-handed knots.
The singular Alexander invariant also conforms to the expected symmetries regarding the chosen singular crossings and handedness. 

From the table, it is clear that although the $\rho_1^\textup{s}$ invariant invariant can distinguish more different knot classes, such as the left-and right-handed knots also for $s=1$ (i.e. $\rho_1$ in ref.~\cite{bar2021perturbed}).
This is specifically the case for singular knots, where the presence of singular crossings increasingly results in a longer singular Alexander polynomial, but also in very lengthy expressions for $\rho_1^\textup{s}$. 
Although $\rho_1^\textup{s}$ is easy to compute and distinguish computationally, the knot classes it finds are not as easy to distinguish from simply looking at the expressions.

The zeroth order in the expansion of the perturbed invariant always gives an expression that is proportional to $\Gamma_\textup{s}$ (but normalised differently). Therefore, we left it out of the table as a separate entry.
As a final note, from the table, one can see that whenever $s=1$, the singular Alexander polynomial ($\rho_1^\textup{s}$) reduces to the Alexander polynomial ($\rho_1$) as if the singular crossings were replaced by positive ones.

We introduce a naming system for the knots: 
\begin{equation}\label{eq:knot_table_naming}
    \textup{(number of crossings)}_{\textup{number}}^{\textup{number of singular crossings}}(\textup{label(s) of the singular crossing(s)}) \nonumber
\end{equation}
The handedness of the knot is given by no bar (right-handed) or a bar (left-handed). 
By handedness, we mean that the knots illustrated in fig.~ \ref{fig:knottab_naming} are right-handed, and if all crossing signs were inverse, then the knots would be left-handed.
If it matters which crossing is considered singular, then this is indicated in brackets, following the labelling of the crossings as in fig.~\ref{fig:knottab_naming}.

\begin{figure}[!ht]
    \captionsetup[subfigure]{justification=centering}
    \centering
    \hfill
    \begin{subfigure}{0.2\textwidth}
        \includegraphics[width=0.8\textwidth]{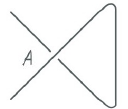}
        \caption{$1_1$}
    \end{subfigure}
    \hfill
    \begin{subfigure}{0.2\textwidth}
        \includegraphics[width=0.7\textwidth]{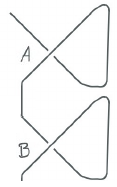}
        \caption{$2_1$}
    \end{subfigure}
    \hfill
    \begin{subfigure}{0.2\textwidth}
        \includegraphics[width=0.5\textwidth]{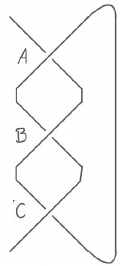}
        \caption{$3_1$}
    \end{subfigure}
    \hfill
    \newline
    \centering
    \hfill
    \begin{subfigure}{0.3\textwidth}
        \includegraphics[width=0.5\textwidth]{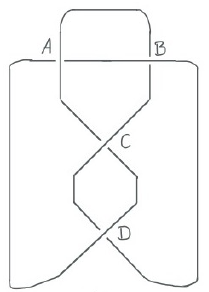}
        \caption{$4_1$}
    \end{subfigure}
    \hfill
    \begin{subfigure}{0.3\textwidth}
        \includegraphics[width=0.5\textwidth]{appendix-d/knot-table-figs-4.pdf}
        \caption{$5_1$}
    \end{subfigure}
    \hfill
    \begin{subfigure}{0.3\textwidth}
        \includegraphics[width=0.5\textwidth]{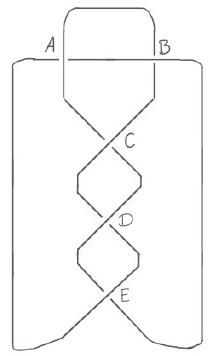}
        \caption{$5_2$}
    \end{subfigure}
    \hfill
    \caption{Labelling of the crossings used in the naming in table~\ref{table:knot_table}. 
    These are the right-handed diagrams. 
    The left-handed versions have the same labelling but the crossing signs are flipped.}
    \label{fig:knottab_naming}
\end{figure}

\newgeometry{left=2mm,bottom=2mm,right=2mm,top=2mm}
 
\begin{center}
\begin{longtable}{||p{20mm} p{72mm}|| p{20mm} p{72mm}||} 
 \hline
 Knot name & $\color{teal}\Gamma_\text{s}$ $\doteq \Delta_K^s$ \newline $\color{magenta}\rho_1^\text{s}$ 
 & Knot name & $\color{teal}\Gamma_\text{s}$ $\doteq \Delta_K^s$\newline $\color{magenta}\rho_1^\text{s}$ \\
 \hline \hline \\[-1em]
 $ 1_{1}^0, \overline{1}_1^0$ & $\color{teal} 1 $ \newline $\color{magenta}0$  &
 $ 1_{1}^1, \overline{1}_1^1$ & $\color{teal} 1 $ \newline $\color{magenta}-A(s-1) s (s t-1) \left[s (t+1)^2+t-1\right]$ \\
 \hline \hline \\[-1em]
 $ 2_{1}^0, \overline{2}_1^0$ & $\color{teal} 1 $ \newline $\color{magenta}0$  &
  \\
 \hline \\[-1em]
 $ 2_{1}^1, \overline{2}_1^1$ & $\color{teal} 1 $ \newline $\color{magenta} -A(s-1) s (s t-1) \left[s (t+1)^2+t-1\right]$  &
 $ 2_{1}^2, \overline{2}_1^2$& $\color{teal} 1 $ \newline $\color{magenta} A \, 2 s^3 t^3 [s^3 t (t+1)^2+s \left(-3 t^2+3 t+4\right)$ \newline $\color{magenta}- s^2 \left(t^4-2 t^3+2 t^2+5 t+2\right)+t^2+t-2]$  \\ 
 \hline \hline \\[-1em]
 $ 3_{1}^0, \overline{3}_1^0$& $\color{teal}t^2-t+1$ \newline $\color{magenta}\pm(t-1)^2 \left(t^2+1\right)t^{-2}$  &
 $ \overline{3}_1^1$ & $\color{teal}t+s-1$ \newline $\color{magenta} -A (s-1) s t (s^2 t^4+s^2 t^3+s^2 t+s^2-s t^2$\newline$\color{magenta}-s t-2 s-t^2+1)$  \\ 
 \hline \\[-1em]
 $ 3_{1}^1$ & $\color{teal}st^2-t+1$ \newline $\color{magenta}A\,t^{-2} \big[s^4 t^3 (t+1)^2 (t^2-t-1)+s^3 (-4 t^6+9 t^4 $\newline$\color{magenta}+4 t^3-3 t^2+t+1)+s^2 (6 t^5-6 t^4-9 t^3+6 t^2+2 t-3)-s (t-1)^2 (4 t^2-3)+(t-1)^3\big]$  &
 $ \overline{3}_1^2$ & $\color{teal} s^2 t+s^2-2 s +1$ \newline $\color{magenta}-A\,t^{-2}((s-1) (s^5 (t+1)^3+s^4 (t+1)^2 (2 t^2-t-5)+s^3 (2 t^5+2 t^4-7 t^3-3 t^2+14 t+10)-s^2 (t^3-7 t^2+6 t+10)-s (2 t^3+2 t^2+t-5)+t-1))$ \\
 \hline \\[-1em]
 $ 3_{1}^2$& $\color{teal}s^2 t^2+s^2 t-2 s t+1$ \newline $\color{magenta}A\,t^{-2}[(s t-1) (s^5 t (t+1)^3 (t^2-2)-s^4 t (t+1)^2 (5 t^2-t-8)+s^3 (10 t^4+8 t^3-17 t^2-17 t-2)+s^2 (-10 t^3+2 t^2+13 t+5)+s (5 t^2-t-4)-t+1)]$ &
 $ 3_{1}^3, \overline{3}_1^3$ & $\color{teal}s^2 t^2+2 s^2 t-3 s t+s^2 -3 s +3 $ \newline $\color{magenta} A\,t^{-2} s (s^7 (t+1)^5 (t^2-2)-8 s^6 (t+1)^4 (t^2-2)+s^5 (t+1)^3 (28 t^2-3 t-53)-s^4 (t+1)^2 (53 t^2-15 t-94)+s^3 (52 t^3+22 t^2-125 t-95)+s^2 (3 t^3-23 t^2+33 t+55)-18 s-3 t+3)$ \\
 \hline \hline \\[-1em]
 $ 4_{1}^0, \overline{4}_1^0$ & $\color{teal}-t^2+3 t-1$ \newline $\color{magenta}0$ &
 $  4_{1}^1(C), 4_{1}^1(D), \newline \overline{4}_{1}^1(A), \overline{4}_{1}^1(B)$ & $\color{teal} 2 s t-s-t+1$ \newline $\color{magenta}-A\,t^{-2}(s-1) (s^3 t^7-3 s^3 t^6-3 s^3 t^5+8 s^3 t^4+5 s^3 t^3-4 s^3 t^2-s^3 t+s^3+8 s^2 t^4-16 s^2 t^3+2 s^2 t^2+5 s^2 t-3 s^2-s t^5+2 s t^4-2 s t^3+5 s t^2-7 s t+3 s+t^3-3 t^2+3 t-1)$ \\
 \hline \\[-1em]
 $  4_{1}^1(A), 4_{1}^1(B), \newline \overline{4}_{1}^1(C), \overline{4}_{1}^1(D)$ & $\color{teal} -s t^2+2 s t+t-1$ \newline $\color{magenta}-A\,t^{-2}(s-1) s (4 s^2 t^5+3 s^2 t^4-4 s^2 t^3-s^2 t^2+2 s^2 t-10 s t^3+8 s t^2-s t-s-4 t^3+7 t^2-4 t+1)$ &
 $  4_{1}^2(AB), \newline \overline{4}_{1}^2(CD)$ & $\color{teal} s -t^2+s t+s+2 t-2$ \newline $\color{magenta} -A\,t^{-2}(s-1) s (3 s^4 t^5+5 s^4 t^4-2 s^4 t^2+s^4 t+s^4+3 s^3 t^5-s^3 t^4-10 s^3 t^3+2 s^3 t^2+s^3 t-3 s^3+2 s^2 t^5+2 s^2 t^4-13 s^2 t^3+5 s^2 t^2+s^2 t+3 s^2-11 s t^3+17 s t^2-3 s t-3 s-2 t^3+6 t^2-6 t+2)$ \\
 \hline \\[-1em]
 $ 4_{1}^2(CD), \newline \overline{4}_{1}^2(AB)$ & $\color{teal}s t^2+s t-s-2 t+2$ \newline $\color{magenta}A\,t^{-2}s (s^5 t^7+3 s^5 t^6-4 s^5 t^5-12 s^5 t^4-s^5 t^3+7 s^5 t^2-2 s^5-6 s^4 t^6+28 s^4 t^4+14 s^4 t^3-26 s^4 t^2-4 s^4 t+10 s^4+6 s^3 t^5-10 s^3 t^4-37 s^3 t^3+33 s^3 t^2+19 s^3 t-19 s^3+2 s^2 t^5-10 s^2 t^4+26 s^2 t^3-8 s^2 t^2-28 s^2 t+18 s^2+3 s t^3-13 s t^2+19 s t-9 s-2 t^3+6 t^2-6 t+2)$ &
 $ 4_{1}^2(AC \text{ or } BD),$ \newline $ 4_{1}^2(BC \text{ or } AD),$ \newline $ \overline{4}_{1}^2(AC\text{ or }BD),$ \newline $ \overline{4}_{1}^2(BC\text{ or }AD)$ & $\color{teal}2 s^2 t-s t-s+1$ \newline $\color{magenta} -A\,t^{-2}(s-1) (s^5 t^7-s^5 t^6-3 s^5 t^5+s^5 t^4+4 s^5 t^3+4 s^5 t^2+2 s^5 t-2 s^4 t^6+16 s^4 t^4-7 s^4 t^3-13 s^4 t^2-9 s^4 t-s^4+s^3 t^4-14 s^3 t^3+12 s^3 t^2+15 s^3 t+4 s^3+2 s^2 t^4+3 s^2 t^3+s^2 t^2-10 s^2 t-6 s^2-s t^3-4 s t^2+s t+4 s+t-1)$ \\
 \hline \\[-1em]
 $ 4_{1}^3(ABC), \newline 4_{1}^3(ABD),$ \newline $ \overline{4}_{1}^3(BCD), \newline \overline{4}_{1}^3(ACD)$ & $\color{teal}2 s^2 t+s^2-s t-3 s+2$ \newline $\color{magenta}-A\,t^{-2}(s-1) s (s^6 t^7+s^6 t^6-s^6 t^5+6 s^6 t^3+11 s^6 t^2+8 s^6 t+2 s^6-4 s^5 t^6-4 s^5 t^5+12 s^5 t^4-5 s^5 t^3-35 s^5 t^2-36 s^5 t-12 s^5+3 s^4 t^5+6 s^4 t^4-21 s^4 t^3+34 s^4 t^2+65 s^4 t+29 s^4+4 s^3 t^4+4 s^3 t^3-57 s^3 t-37 s^3-2 s^2 t^3-10 s^2 t^2+21 s^2 t+27 s^2+s t-11 s-2 t+2)$ &
 $ 4_{1}^3(ACD), \newline 4_{1}^3(BCD),$ \newline $ \overline{4}_{1}^3(ABC), \newline \overline{4}_{1}^3(ABD)$ & $\color{teal}s^2 t^2+2 s^2 t-3 s t-s+2$ \newline $\color{magenta} A\,t^{-2}s (s^7 t^7+5 s^7 t^6+3 s^7 t^5-10 s^7 t^4-16 s^7 t^3-9 s^7 t^2-2 s^7 t-8 s^6 t^6-20 s^6 t^5+5 s^6 t^4+60 s^6 t^3+55 s^6 t^2+19 s^6 t+s^6+19 s^5 t^5+33 s^5 t^4-64 s^5 t^3-117 s^5 t^2-63 s^5 t-8 s^5+2 s^4 t^5-28 s^4 t^4-2 s^4 t^3+109 s^4 t^2+98 s^4 t+23 s^4-4 s^3 t^4+25 s^3 t^3-37 s^3 t^2-75 s^3 t-31 s^3-6 s^2 t^2+30 s^2 t+22 s^2+4 s t^2-5 s t-9 s-2 t+2)$ \\
 \hline \\[-1em]
 $  4_{1}^4$ & $\color{teal}s^2 t^2+3 s^2 t+s^2-4 s t-4 s+4$ \newline $\color{magenta} A \, t^{-2} s^3 (s^7 t^7+7 s^7 t^6+9 s^7 t^5-9 s^7 t^4-33 s^7 t^3-35 s^7 t^2-17 s^7 t-3 s^7-10 s^6 t^6-36 s^6 t^5-14 s^6 t^4+116 s^6 t^3+186 s^6 t^2+116 s^6 t+26 s^6+27 s^5 t^5+63 s^5 t^4-118 s^5 t^3-386 s^5 t^2-325 s^5 t-93 s^5+4 s^4 t^5-28 s^4 t^4+12 s^4 t^3+392 s^4 t^2+488 s^4 t+180 s^4-16 s^3 t^4+14 s^3 t^3-178 s^3 t^2-426 s^3 t-210 s^3+12 s^2 t^3+4 s^2 t^2+204 s^2 t+156 s^2+16 s t^2-24 s t-72 s-16 t+16)$
 & \\
 \hline \hline \\[-1em]
 $  5_{1}^0, \overline{5}_1^0$ & $\color{teal} t^4-t^3+t^2-t+1$ \newline $\color{magenta}\pm t^{-4}(t-1)^2 \left(t^2+1\right) \left(2 t^4+t^2+2\right)$ &
 \\
 \hline \\[-1em]
 $  5_{1}^1$ & $\color{teal} s t^4-t^3+t^2-t+1$ \newline $\color{magenta}A\,t^{-4}[2 s^4 t^{11}+2 s^4 t^{10}-3 s^4 t^9-3 s^4 t^8-s^4 t^6+s^4 t^4-s^4 t^3-s^4 t^2-8 s^3 t^{10}+12 s^3 t^8+s^3 t^6+3 s^3 t^5-4 s^3 t^4+s^3 t^3+s^3 t^2+s^3 t+s^3+12 s^2 t^9-12 s^2 t^8-6 s^2 t^7+6 s^2 t^6-9 s^2 t^5+6 s^2 t^4+3 s^2 t-4 s^2-8 s t^8+16 s t^7-12 s t^6+8 s t^5-s t^4-5 s t^3+7 s t^2-10 s t+5 s+2 t^7-6 t^6+9 t^5-11 t^4+11 t^3-9 t^2+6 t-2]$ 
 &
 $  \overline{5}_{1}^1$ & $\color{teal} s+t^3-t^2+t-1$ \newline $\color{magenta}A\,t^{-4}[1 - 4 s + 6 s^2 - 4 s^3 + s^4 - 3 t + 8 s t - 6 s^2 t + s^4 t + 
 4 t^2 - 4 s t^2 - 6 s^2 t^2 + 8 s^3 t^2 - 2 s^4 t^2 - 4 t^3 + 
 6 s^2 t^3 - 2 s^4 t^3 + 4 t^4 - 4 s^3 t^4 + s^4 t^4 - 4 t^5 + 
 s t^5 - 3 s^2 t^5 + 3 s^3 t^5 + 3 t^6 + s t^6 + s^3 t^6 - 
 s^4 t^6 - t^7 - 3 s t^7 + 2 s t^8 + s^3 t^8 - s t^9 - s^2 t^9 + 
 s^3 t^9 + s^3 t^10 - s^4 t^10 + s^3 t^11 - s^4 t^11]$ 
 \\
 \hline \\[-1em]
 $  5_{1}^2$ & $\color{teal} s^2 t^4+s^2 t^3-2 s t^3+t^2-t+1$ \newline $\color{magenta}A\,t^{-4}[2 s^6 t^{11}+6 s^6 t^{10}+3 s^6 t^9-7 s^6 t^8-9 s^6 t^7-3 s^6 t^6-12 s^5 t^{10}-24 s^5 t^9+6 s^5 t^8+36 s^5 t^7+16 s^5 t^6-2 s^5 t^5+2 s^5 t^4+2 s^5 t^3+30 s^4 t^9+30 s^4 t^8-45 s^4 t^7-43 s^4 t^6+10 s^4 t^5-2 s^4 t^4-12 s^4 t^3-2 s^4 t^2-40 s^3 t^8+60 s^3 t^6-8 s^3 t^5-12 s^3 t^4+22 s^3 t^3+2 s^3 t^2+2 s^3 t+2 s^3+30 s^2 t^7-30 s^2 t^6-15 s^2 t^5+27 s^2 t^4-19 s^2 t^3-s^2 t^2+4 s^2 t-6 s^2-12 s t^6+24 s t^5-18 s t^4+4 s t^3+8 s t^2-12 s t+6 s+2 t^5-6 t^4+9 t^3-9 t^2+6 t-2]$ 
 &
 $  \overline{5}_{1}^2$ & $\color{teal} s^2 t+s^2-2 s+t^2-t+1$ \newline $\color{magenta} -A\,t^{-2}(s-1) (s^5 t^3+3 s^5 t^2+3 s^5 t+s^5+2 s^4 t^4+3 s^4 t^3-5 s^4 t^2-11 s^4 t-5 s^4+2 s^3 t^7+2 s^3 t^6-7 s^3 t^3-3 s^3 t^2+14 s^3 t+10 s^3-2 s^2 t^5-2 s^2 t^4+s^2 t^3+9 s^2 t^2-6 s^2 t-10 s^2-2 s t^5+2 s t^3-4 s t^2-s t+5 s+t-1)$ 
 \\
 \hline \\[-1em]
 $  5_{1}^3$ & $\color{teal} s^3 t^4+2 s^3 t^3+s^3 t^2-3 s^2 t^3-3 s^2 t^2+3 s t^2-t+1$ \newline $\color{magenta}A\,t^{-4}(s t-1) (2 s^7 t^{10}+10 s^7 t^9+17 s^7 t^8+5 s^7 t^7-20 s^7 t^6-28 s^7 t^5-15 s^7 t^4-3 s^7 t^3-14 s^6 t^9-54 s^6 t^8-55 s^6 t^7+37 s^6 t^6+108 s^6 t^5+68 s^6 t^4+9 s^6 t^3-3 s^6 t^2+42 s^5 t^8+114 s^5 t^7+29 s^5 t^6-162 s^5 t^5-150 s^5 t^4-16 s^5 t^3+15 s^5 t^2-70 s^4 t^7-110 s^4 t^6+88 s^4 t^5+195 s^4 t^4+33 s^4 t^3-37 s^4 t^2-3 s^4 t+70 s^3 t^6+30 s^3 t^5-137 s^3 t^4-60 s^3 t^3+45 s^3 t^2+5 s^3 t-3 s^3-42 s^2 t^5+30 s^2 t^4+61 s^2 t^3-27 s^2 t^2-12 s^2 t+8 s^2+14 s t^4-26 s t^3+3 s t^2+16 s t-7 s-2 t^3+6 t^2-6 t+2)$
 &
 $  \overline{5}_{1}^3$ & $\color{teal} s^3 t^2+2 s^3 t+s^3-3 s^2 t-3 s^2+3 s+t-1$ \newline $\color{magenta}-A\,t^{-4}(s-1) (s^7 t^5+5 s^7 t^4+10 s^7 t^3+10 s^7 t^2+5 s^7 t+s^7+s^6 t^5-3 s^6 t^4-22 s^6 t^3-38 s^6 t^2-27 s^6 t-7 s^6+3 s^5 t^6+7 s^5 t^5-3 s^5 t^4+43 s^5 t^2+57 s^5 t+21 s^5-14 s^4 t^5-18 s^4 t^4+21 s^4 t^3+5 s^4 t^2-55 s^4 t-35 s^4+3 s^3 t^7+3 s^3 t^6+s^3 t^5+27 s^3 t^4+6 s^3 t^3-40 s^3 t^2+15 s^3 t+35 s^3-2 s^2 t^5-6 s^2 t^4-24 s^2 t^3+20 s^2 t^2+15 s^2 t-21 s^2-3 s t^5-2 s t^4+8 s t^3+3 s t^2-13 s t+7 s+t^3-3 t^2+3 t-1)$ 
 \\
 \hline \\[-1em]
 $  5_{1}^4$ & $\color{teal} s^4 t^4+3 s^4 t^3+3 s^4 t^2+s^4 t-4 s^3 t^3-8 s^3 t^2-4 s^3 t+6 s^2 t^2+6 s^2 t-4 s t+1$ \newline $\color{magenta}A\,t^{-4}[2 s^{10} t^{11}+14 s^{10} t^{10}+39 s^{10} t^9+49 s^{10} t^8+7 s^{10} t^7-63 s^{10} t^6-91 s^{10} t^5-61 s^{10} t^4-21 s^{10} t^3-3 s^{10} t^2-20 s^9 t^{10}-120 s^9 t^9-270 s^9 t^8-220 s^9 t^7+150 s^9 t^6+480 s^9 t^5+430 s^9 t^4+180 s^9 t^3+30 s^9 t^2+90 s^8 t^9+450 s^8 t^8+765 s^8 t^7+225 s^8 t^6-900 s^8 t^5-1260 s^8 t^4-675 s^8 t^3-135 s^8 t^2-240 s^7 t^8-960 s^7 t^7-1084 s^7 t^6+468 s^7 t^5+1912 s^7 t^4+1448 s^7 t^3+372 s^7 t^2+4 s^7 t+420 s^6 t^7+1264 s^6 t^6+666 s^6 t^5-1414 s^6 t^4-1898 s^6 t^3-690 s^6 t^2-28 s^6 t-504 s^5 t^6-1032 s^5 t^5+144 s^5 t^4+1452 s^5 t^3+864 s^5 t^2+84 s^5 t+420 s^4 t^5+480 s^4 t^4-494 s^4 t^3-694 s^4 t^2-140 s^4 t-240 s^3 t^4-76 s^3 t^3+304 s^3 t^2+144 s^3 t+4 s^3+90 s^2 t^3-30 s^2 t^2-78 s^2 t-10 s^2-20 s t^2+12 s t+8 s+2 t-2]$ 
 &
 $  \overline{5}_{1}^4$ & $\color{teal} s^4 t^3+3 s^4 t^2+3 s^4 t+s^4-4 s^3 t^2-8 s^3 t-4 s^3+6 s^2 t+6 s^2-4 s+1$ \newline $\color{magenta}A\,t^{-4}[s^{10} t^9+7 s^{10} t^8+19 s^{10} t^7+21 s^{10} t^6-7 s^{10} t^5-49 s^{10} t^4-63 s^{10} t^3-41 s^{10} t^2-14 s^{10} t-2 s^{10}-10 s^9 t^8-60 s^9 t^7-130 s^9 t^6-80 s^9 t^5+150 s^9 t^4+340 s^9 t^3+290 s^9 t^2+120 s^9 t+20 s^9+45 s^8 t^7+225 s^8 t^6+360 s^8 t^5-675 s^8 t^3-855 s^8 t^2-450 s^8 t-90 s^8-124 s^7 t^6-492 s^7 t^5-488 s^7 t^4+488 s^7 t^3+1332 s^7 t^2+964 s^7 t+240 s^7+4 s^6 t^6+246 s^6 t^5+686 s^6 t^4+202 s^6 t^3-1110 s^6 t^2-1288 s^6 t-420 s^6-24 s^5 t^5-360 s^5 t^4-564 s^5 t^3+360 s^5 t^2+1092 s^5 t+504 s^5-4 s^4 t^5+56 s^4 t^4+350 s^4 t^3+150 s^4 t^2-560 s^4 t-420 s^4+4 s^3 t^5+4 s^3 t^4-76 s^3 t^3-176 s^3 t^2+140 s^3 t+240 s^3+2 s^2 t^3+54 s^2 t^2+6 s^2 t-90 s^2-4 s t^3-4 s t^2-12 s t+20 s+2 t-2]$
 \\
 \hline \\[-1em]
 $  5_1^5$ & $\color{teal} s^4 t^4+4 s^4 t^3+6 s^4 t^2+4 s^4 t+s^4-5 s^3 t^3-15 s^3 t^2-15 s^3 t-5 s^3+10 s^2 t^2+20 s^2 t+10 s^2-10 s t-10 s+5$ \newline $\color{magenta}A\,t^{-4}s (2 s^{11} t^{11}+18 s^{11} t^{10}+69 s^{11} t^9+141 s^{11} t^8+144 s^{11} t^7-210 s^{11} t^5-306 s^{11} t^4-234 s^{11} t^3-106 s^{11} t^2-27 s^{11} t-3 s^{11}-24 s^{10} t^{10}-192 s^{10} t^9-636 s^{10} t^8-1056 s^{10} t^7-672 s^{10} t^6+672 s^{10} t^5+1848 s^{10} t^4+1824 s^{10} t^3+984 s^{10} t^2+288 s^{10} t+36 s^{10}+132 s^9 t^9+924 s^9 t^8+2574 s^9 t^7+3234 s^9 t^6+462 s^9 t^5-4158 s^9 t^4-6006 s^9 t^3-4026 s^9 t^2-1386 s^9 t-198 s^9-440 s^8 t^8-2640 s^8 t^7-5940 s^8 t^6-4840 s^8 t^5+3300 s^8 t^4+10560 s^8 t^3+9460 s^8 t^2+3960 s^8 t+660 s^8+990 s^7 t^7+4945 s^7 t^6+8395 s^7 t^5+2450 s^7 t^4-9900 s^7 t^3-13835 s^7 t^2-7405 s^7 t-1480 s^7-1579 s^6 t^6-6281 s^6 t^5-6998 s^6 t^4+3238 s^6 t^3+12577 s^6 t^2+9379 s^6 t+2336 s^6+1813 s^5 t^5+5334 s^5 t^4+2492 s^5 t^3-6398 s^5 t^2-8001 s^5 t-2632 s^5-1479 s^4 t^4-2783 s^4 t^3+967 s^4 t^2+4367 s^4 t+2096 s^4+810 s^3 t^3+635 s^3 t^2-1310 s^3 t-1135 s^3+5 s^2 t^3-260 s^2 t^2+110 s^2 t+385 s^2+30 s t-70 s-5 t+5)$ & \\
 \hline \hline \\[-1em]
 $  5_2^0$ & $\color{teal} 2 t^2-3 t+2$ \newline $\color{magenta} t^{-2}(t-1)^2 \left(5 t^2-4 t+5\right)$ &
$ \overline{5}_{2}^0$ & $\color{teal} 2 t^2-3 t+2$ \newline $\color{magenta}- t^{-2}(t-1)^2 \left(5 t^2-4 t+5\right)$
 \\
 \hline \\[-1em]
 $  5_2^1(A), \newline 5_2^1(B)$ & $\color{teal} 2 s t^2-s t-2 t+2$ \newline $\color{magenta}A\,t^{-2}(5 s^4 t^7+s^4 t^6-11 s^4 t^5-2 s^4 t^4+6 s^4 t^3-s^4 t^2-2 s^4 t-20 s^3 t^6+13 s^3 t^5+32 s^3 t^4-10 s^3 t^3-21 s^3 t^2+10 s^3 t+4 s^3+30 s^2 t^5-42 s^2 t^4-25 s^2 t^3+43 s^2 t^2+3 s^2 t-13 s^2-20 s t^4+43 s t^3-11 s t^2-26 s t+14 s+5 t^3-15 t^2+15 t-5)$ &
 $  \overline{5}_2^1(A), \newline \overline{5}_2^1(B)$ & $\color{teal} -s t+2 s+2 t-2$ \newline $\color{magenta}-A\,t^{-2}(s-1) (4 s^3 t^7-4 s^3 t^5+3 s^3 t^4+s^3 t^3+s^3 t-s^3-5 s^2 t^5+3 s^2 t^4-s^2 t^3+s^2 t^2-5 s^2 t+3 s^2-4 s t^5+6 s t^4-2 s t^3-4 s t^2+7 s t-3 s-t^3+3 t^2-3 t+1)$
 \\
 \hline \\[-1em]
 $ 5_2^1(C), \newline 5_2^1(D), \newline 5_2^1(E)$ & $\color{teal} 2 s t^2-2 s t+s-t+1$ \newline $\color{magenta} A t^{-2} (4 s^4 t^7-2 s^4 t^6-8 s^4 t^5+3 s^4 t^4+2 s^4 t^3-3 s^4 t^2+s^3 t^7-17 s^3 t^6+27 s^3 t^5-4 s^3 t^4-3 s^3 t^3+4 s^3 t+14 s^2 t^5-29 s^2 t^4+12 s^2 t^3+3 s^2 t^2-3 s^2 t-s^2-s t^5-2 s t^4+7 s t^3-2 s t^2-4 s t+2 s+t^3-3 t^2+3 t-1)$ &
 $ \overline{5}_2^1(C),\newline \overline{5}_2^1(D), \newline \overline{5}_2^1(E)$ & $\color{teal} s t^2-2 s t+2 s+t-1$ \newline $\color{magenta}-A\,t^{-2} s (5 s^3 t^7-5 s^3 t^6-3 s^3 t^5+10 s^3 t^4-5 s^3 t^3+5 s^3 t-3 s^3-4 s^2 t^7+2 s^2 t^6+2 s^2 t^5-3 s^2 t^4-10 s^2 t^3+15 s^2 t^2-17 s^2 t+7 s^2+s t^5-3 s t^4+12 s t^3-15 s t^2+14 s t-5 s+4 t^5-8 t^4+6 t^3-t^2-2 t+1)$
 \\
 \hline \\[-1em]
 $  5_2^2(CD), \newline 5_2^2(DE), \newline 5_2^2(CE)$ & $\color{teal} 2 s t^2-s t-2 t+2$ \newline $\color{magenta} A \,t^{-2}s (3 s^5 t^7+s^5 t^6-10 s^5 t^5-6 s^5 t^4+5 s^5 t^3+s^5 t^2-2 s^5 t+2 s^4 t^7-16 s^4 t^6+18 s^4 t^5+22 s^4 t^4-20 s^4 t^2+6 s^4 t+4 s^4-4 s^3 t^6+24 s^3 t^5-46 s^3 t^4-s^3 t^3+25 s^3 t^2+5 s^3 t-11 s^3-6 s^2 t^4+28 s^2 t^3-22 s^2 t^2-12 s^2 t+12 s^2+4 s t^4-11 s t^3+5 s t^2+9 s t-7 s-2 t^3+6 t^2-6 t+2)$ &
 $  5_2^2(CD), \newline 5_2^2(DE), \newline 5_2^2(CE)$ & $\color{teal} -s t+2 s+2 t-2$ \newline $\color{magenta} -A\,t^{-2}(s-1) s (2 s^4 t^7-2 s^4 t^6+s^4 t^5+7 s^4 t^4-2 s^4 t^3+3 s^4 t-s^4+4 s^3 t^6-5 s^3 t^5-s^3 t^4-6 s^3 t^3+10 s^3 t^2-11 s^3 t+s^3-2 s^2 t^4+7 s^2 t^3-15 s^2 t^2+7 s^2 t+3 s^2-4 s t^4+3 s t^3-s t^2+7 s t-5 s-2 t^3+6 t^2-6 t+2)$
 \\
 \hline \\[-1em]
 $  5_2^2(AB)$ & $\color{teal} 2 s^2 t^2+s^2 t-s^2-4 s t+2 s+1$ \newline $\color{magenta} A\,t^{-2}(5 s^6 t^7+11 s^6 t^6-5 s^6 t^5-23 s^6 t^4-6 s^6 t^3+10 s^6 t^2+2 s^6 t-2 s^6-30 s^5 t^6-32 s^5 t^5+60 s^5 t^4+62 s^5 t^3-26 s^5 t^2-18 s^5 t+8 s^5+67 s^4 t^5+13 s^4 t^4-126 s^4 t^3-22 s^4 t^2+42 s^4 t-8 s^4+2 s^3 t^5-82 s^3 t^4+50 s^3 t^3+78 s^3 t^2-20 s^3 t+41 s^2 t^3-43 s^2 t^2-9 s^2 t+s^2-2 s t^3-2 s t^2+2 s t+2 s+t-1)$ &
 $  5_2^2(AB)$ & $\color{teal} s^2 \left(-t^2\right)+s^2 t+2 s^2+2 s t-4 s+1$ \newline $\color{magenta}-A\,t^{-2}(s-1) (2 s^5 t^5+2 s^5 t^4-3 s^5 t^3-s^5 t^2+5 s^5 t+3 s^5-2 s^4 t^5-4 s^4 t^4+7 s^4 t^3+7 s^4 t^2-13 s^4 t-11 s^4+8 s^3 t^5+8 s^3 t^4-15 s^3 t^3-15 s^3 t^2+16 s^3 t+16 s^3-17 s^2 t^3+27 s^2 t^2-6 s^2 t-14 s^2-8 s t^3+10 s t^2-9 s t+7 s+t-1)$ \\
 \hline \\[-1em]
 $ 5_2^2(BC),\newline 5_2^2(AC),\newline 5_2^2(BE),\newline 5_2^2(AD),\newline 5_2^2(AE),\newline 5_2^2(BD)$ & $\color{teal} 2 s^2 t^2-3 s t+s+1$ \newline $\color{magenta}A\,t^{-2}(4 s^6 t^7+6 s^6 t^6-5 s^6 t^5-11 s^6 t^4-3 s^6 t^3+s^6 t^2+s^5 t^7-23 s^5 t^6-13 s^5 t^5+29 s^5 t^4+35 s^5 t^3-3 s^5 t^2-3 s^5 t+s^5-2 s^4 t^6+50 s^4 t^5-3 s^4 t^4-56 s^4 t^3-32 s^4 t^2+9 s^4 t-49 s^3 t^4+23 s^3 t^3+47 s^3 t^2+9 s^3 t-2 s^3+2 s^2 t^4+21 s^2 t^3-14 s^2 t^2-18 s^2 t-s^2-s t^3-4 s t^2+2 s t+3 s+t-1)$ &
 $  \overline{5}_2^2(BC),\newline \overline{5}_2^2(AC),\newline \overline{5}_2^2(BE),\newline \overline{5}_2^2(AD),\newline $ & $\color{teal} 2 s^2+s t-3 s+1$ \newline $\color{magenta} -A\,t^{-2}((s-1) (s^5 t^7-s^5 t^6+2 s^5 t^4-2 s^5 t^3+2 s^5 t^2+5 s^5 t+s^5+2 s^4 t^6-5 s^4 t^5-s^4 t^4+7 s^4 t^3+5 s^4 t^2-18 s^4 t-6 s^4+4 s^3 t^5+3 s^3 t^4-s^3 t^3-19 s^3 t^2+18 s^3 t+13 s^3-2 s^2 t^4-2 s^2 t^3+8 s^2 t^2-s^2 t-13 s^2-5 s t^3+4 s t^2-5 s t+6 s+t-1)$ \\
 \hline \\[-1em]
 $  \overline{5}_2^2(BD),\newline \overline{5}_2^2(AE) $ & $\color{teal} s^2 t+2 s^2-4 s+2$ \newline $\color{magenta} -A\,t^{-2}(s-1) s (s^6 t^5+2 s^6 t^3+11 s^6 t^2+11 s^6 t+3 s^6-4 s^5 t^5+4 s^5 t^4+8 s^5 t^3-26 s^5 t^2-49 s^5 t-17 s^5+6 s^4 t^5-2 s^4 t^4-14 s^4 t^3+6 s^4 t^2+80 s^4 t+40 s^4+4 s^3 t^4-s^3 t^3+15 s^3 t^2-54 s^3 t-50 s^3-4 s^2 t^3-2 s^2 t^2+7 s^2 t+35 s^2-4 s t^2+7 s t-13 s-2 t+2)$ & \\
 \hline \\[-1em]
 $  5_2^3(ABC), \newline 5_2^3(ABD), \newline 5_2^3(ABE)$ & $\color{teal} 2 s^2 t^2+2 s^2 t-5 s t-s+3$ \newline $\color{magenta} A\,t^{-2} s (4 s^7 t^7+14 s^7 t^6+10 s^7 t^5-16 s^7 t^4-28 s^7 t^3-14 s^7 t^2-2 s^7 t+s^6 t^7-29 s^6 t^6-71 s^6 t^5-5 s^6 t^4+107 s^6 t^3+97 s^6 t^2+27 s^6 t+s^6-4 s^5 t^6+86 s^5 t^5+122 s^5 t^4-89 s^5 t^3-219 s^5 t^2-109 s^5 t-11 s^5+7 s^4 t^5-129 s^4 t^4-56 s^4 t^3+178 s^4 t^2+185 s^4 t+39 s^4-4 s^3 t^4+90 s^3 t^3-42 s^3 t^2-133 s^3 t-57 s^3-5 s^2 t^3-13 s^2 t^2+47 s^2 t+39 s^2+8 s t^2-12 s t-14 s-3 t+3)$ &
 $\overline{5}_2^3(ABC), \newline \overline{5}_2^3(ABD), \newline \overline{5}_2^3(ABE)$ & $\color{teal} 2 s^2 t+2 s^2-s t-5 s+3$ \newline $\color{magenta}-A\,t^{-2}(s-1) s (s^6 t^7+s^6 t^6-5 s^6 t^5-9 s^6 t^4+3 s^6 t^3+19 s^6 t^2+17 s^6 t+5 s^6-4 s^5 t^6+32 s^5 t^4+28 s^5 t^3-48 s^5 t^2-76 s^5 t-28 s^5+3 s^4 t^5-13 s^4 t^4-65 s^4 t^3+13 s^4 t^2+126 s^4 t+64 s^4+12 s^3 t^4+17 s^3 t^3+41 s^3 t^2-89 s^3 t-77 s^3-s^2 t^3-17 s^2 t^2+16 s^2 t+52 s^2-8 s t^2+9 s t-19 s-3 t+3)$ \\
 \hline \\[-1em]
 $  5_2^3(BCD),\newline 5_2^3(ADE), \newline 5_2^3(BDE)$, \newline $ 5_2^3(ACD), \newline 5_2^3(BCE),\newline 5_2^3(ACE)$ & $\color{teal} 2 s^2 t^2+s^2 t-4 s t+2$ \newline $\color{magenta}A\,t^{-2}s (s t-1) (3 s^6 t^6+7 s^6 t^5-2 s^6 t^4-17 s^6 t^3-15 s^6 t^2-4 s^6 t+2 s^5 t^6-15 s^5 t^5-16 s^5 t^4+30 s^5 t^3+59 s^5 t^2+24 s^5 t-6 s^4 t^5+32 s^4 t^4-5 s^4 t^3-68 s^4 t^2-61 s^4 t-8 s^4+4 s^3 t^4-26 s^3 t^3+36 s^3 t^2+53 s^3 t+19 s^3+4 s^2 t^3-s^2 t^2-23 s^2 t-16 s^2-6 s t^2+9 s t+7 s+2 t-2)$ &
$ \overline{5}_2^3(BCD),\newline \overline{5}_2^3(ADE), \newline \overline{5}_2^3(BDE)$, \newline $ \overline{5}_2^3(ACD), \newline \overline{5}_2^3(BCE),\newline \overline{5}_2^3(ACE)$ & $\color{teal} s^2 t+2 s^2-4 s+2$ \newline $\color{magenta} -A\,t^{-2}(s-1) s (s^6 t^5+2 s^6 t^3+11 s^6 t^2+11 s^6 t+3 s^6-4 s^5 t^5+4 s^5 t^4+8 s^5 t^3-26 s^5 t^2-49 s^5 t-17 s^5+6 s^4 t^5-2 s^4 t^4-14 s^4 t^3+6 s^4 t^2+80 s^4 t+40 s^4+4 s^3 t^4-s^3 t^3+15 s^3 t^2-54 s^3 t-50 s^3-4 s^2 t^3-2 s^2 t^2+7 s^2 t+35 s^2-4 s t^2+7 s t-13 s-2 t+2)$ \\
 \hline \\[-1em]
 $  5_2^3(CDE)$ & $\color{teal} 2 s^2 t^2+2 s^2 t-5 s t-s+3$ \newline $\color{magenta}A\,t^{-2} s (4 s^7 t^7+14 s^7 t^6+10 s^7 t^5-16 s^7 t^4-28 s^7 t^3-14 s^7 t^2-2 s^7 t+s^6 t^7-29 s^6 t^6-71 s^6 t^5-5 s^6 t^4+107 s^6 t^3+97 s^6 t^2+27 s^6 t+s^6-4 s^5 t^6+86 s^5 t^5+122 s^5 t^4-89 s^5 t^3-219 s^5 t^2-109 s^5 t-11 s^5+7 s^4 t^5-129 s^4 t^4-56 s^4 t^3+178 s^4 t^2+185 s^4 t+39 s^4-4 s^3 t^4+90 s^3 t^3-42 s^3 t^2-133 s^3 t-57 s^3-5 s^2 t^3-13 s^2 t^2+47 s^2 t+39 s^2+8 s t^2-12 s t-14 s-3 t+3)$ &
 $  \overline{5}_2^3(CDE)$ & $\color{teal} 2 s^2 t+2 s^2-s t-5 s+3$ \newline $\color{magenta} -A\,t^{-2}(s-1) s (s^6 t^7+s^6 t^6-5 s^6 t^5-9 s^6 t^4+3 s^6 t^3+19 s^6 t^2+17 s^6 t+5 s^6-4 s^5 t^6+32 s^5 t^4+28 s^5 t^3-48 s^5 t^2-76 s^5 t-28 s^5+3 s^4 t^5-13 s^4 t^4-65 s^4 t^3+13 s^4 t^2+126 s^4 t+64 s^4+12 s^3 t^4+17 s^3 t^3+41 s^3 t^2-89 s^3 t-77 s^3-s^2 t^3-17 s^2 t^2+16 s^2 t+52 s^2-8 s t^2+9 s t-19 s-3 t+3)$ \\
 \hline \\[-1em]
 $  5_2^4(ABCD), \newline 5_2^4(ABCE), \newline 5_2^4(ABDE)$ & $\color{teal} 2 s^2 t^2+3 s^2 t+s^2-6 s t-4 s+5$ \newline $\color{magenta} A\,t^{-2} s^3 (3 s^7 t^7+13 s^7 t^6+14 s^7 t^5-16 s^7 t^4-50 s^7 t^3-46 s^7 t^2-19 s^7 t-3 s^7+2 s^6 t^7-20 s^6 t^6-70 s^6 t^5-14 s^6 t^4+178 s^6 t^3+258 s^6 t^2+142 s^6 t+28 s^6-12 s^5 t^6+58 s^5 t^5+116 s^5 t^4-181 s^5 t^3-541 s^5 t^2-422 s^5 t-108 s^5+30 s^4 t^5-86 s^4 t^4+20 s^4 t^3+536 s^4 t^2+642 s^4 t+222 s^4-32 s^3 t^4+58 s^3 t^3-266 s^3 t^2-550 s^3 t-264 s^3-6 s^2 t^3+10 s^2 t^2+278 s^2 t+194 s^2+44 s t^2-45 s t-91 s-26 t+22)$ &
 $  \overline{5}_2^4(ABCD), \newline \overline{5}_2^4(ABCE), \newline \overline{5}_2^4(ABDE)$ & $\color{teal} s^2 t^2+3 s^2 t+2 s^2-4 s t-6 s+5$ \newline $\color{magenta} A\,t^{-2}s^3 (s^7 t^7+7 s^7 t^6+14 s^7 t^5-36 s^7 t^3-52 s^7 t^2-31 s^7 t-7 s^7-10 s^6 t^6-48 s^6 t^5-48 s^6 t^4+100 s^6 t^3+256 s^6 t^2+200 s^6 t+54 s^6+36 s^5 t^5+108 s^5 t^4-53 s^5 t^3-477 s^5 t^2-528 s^5 t-176 s^5+2 s^4 t^5-52 s^4 t^4-46 s^4 t^3+406 s^4 t^2+736 s^4 t+318 s^4-12 s^3 t^4+14 s^3 t^3-132 s^3 t^2-572 s^3 t-352 s^3+24 s^2 t^3-14 s^2 t^2+218 s^2 t+248 s^2+12 s t^2+3 s t-107 s-26 t+22)$ \\
 \hline \\[-1em]
 $ 5_2^4(BCDE), \newline 5_2^4(ACDE)$ & $\color{teal} 2 s^2 t^2+2 s^2 t-5 s t-s+3$ \newline $\color{magenta} A\,t^{-2}s^3 (2 s^7 t^7+4 s^7 t^6-8 s^7 t^5-30 s^7 t^4-32 s^7 t^3-14 s^7 t^2-2 s^7 t+3 s^6 t^7-5 s^6 t^6+s^6 t^5+83 s^6 t^4+153 s^6 t^3+105 s^6 t^2+27 s^6 t+s^6-18 s^5 t^6+3 s^5 t^5-66 s^5 t^4-259 s^5 t^3-275 s^5 t^2-114 s^5 t-11 s^5+36 s^4 t^5-s^4 t^4+201 s^4 t^3+341 s^4 t^2+215 s^4 t+40 s^4-18 s^3 t^4-17 s^3 t^3-246 s^3 t^2-212 s^3 t-63 s^3-27 s^2 t^3+48 s^2 t^2+138 s^2 t+55 s^2+36 s t^2-40 s t-32 s-12 t+10)$ &
 $ \overline{5}_2^4(BCDE), \newline \overline{5}_2^4(ACDE)$ & $\color{teal} 2 s^2 t+2 s^2-s t-5 s+3$ \newline $\color{magenta} -A\,t^{-2}(s-1) s^3 (s^6 t^7+3 s^6 t^6+3 s^6 t^5+5 s^6 t^4+17 s^6 t^3+27 s^6 t^2+19 s^6 t+5 s^6-6 s^5 t^6-12 s^5 t^5-4 s^5 t^4-34 s^5 t^3-104 s^5 t^2-98 s^5 t-30 s^5+8 s^4 t^5+17 s^4 t^4+13 s^4 t^3+131 s^4 t^2+205 s^4 t+78 s^4+6 s^3 t^4-20 s^3 t^3-46 s^3 t^2-206 s^3 t-114 s^3+3 s^2 t^3-8 s^2 t^2+82 s^2 t+99 s^2+10 s t-48 s-12 t+10)$ \\
 \hline \\[-1em]
 $  \overline{5}_2^5$ & $\color{teal} 2 s^2 t^2+4 s^2 t+2 s^2-7 s t-7 s+7$ \newline $\color{magenta} A\,t^{-2} s^5 (2 s^7 t^7+8 s^7 t^6+s^7 t^5-45 s^7 t^4-100 s^7 t^3-98 s^7 t^2-47 s^7 t-9 s^7+3 s^6 t^7-3 s^6 t^6-11 s^6 t^5+123 s^6 t^4+437 s^6 t^3+571 s^6 t^2+339 s^6 t+77 s^6-24 s^5 t^6-35 s^5 t^5-143 s^5 t^4-738 s^5 t^3-1346 s^5 t^2-1019 s^5 t-279 s^5+77 s^4 t^5+141 s^4 t^4+594 s^4 t^3+1638 s^4 t^2+1673 s^4 t+565 s^4-108 s^3 t^4-187 s^3 t^3-985 s^3 t^2-1620 s^3 t-714 s^3+13 s^2 t^3+83 s^2 t^2+841 s^2 t+591 s^2+132 s t^2-74 s t-304 s-93 t+73)$ & \\
 \hline
\caption{  $A = (t-1)^{-1}(s-1+st)^{-2}$. 
The $\pm$ sometimes written in the invariants indicates a plus for the right-handed variant and a minus for the left-handed variant.
The $\Gamma_\textup{s}$ entries are proportional upto a value $t^{\pm k} s^{\pm l}$ ($k,l\in\mathbb{Z}$) to the computational value from $\Gamma_\textup{s}$, to group them more efficiently. 
}
\label{table:knot_table}
\end{longtable}
\end{center}


\end{document}